\documentclass[12pt,a4paper, oneside, reqno]{amsart}
\usepackage[foot]{amsaddr} 
\usepackage{mathtools}
\usepackage{amsmath, nicefrac, amsthm, verbatim, amsfonts, mathtools, amssymb, upgreek, xcolor, bbm}
\usepackage{amsmath, nicefrac, amsthm, verbatim, amsfonts, amssymb, xcolor, bbm}
\usepackage{graphics, xspace, enumerate}
\usepackage{enumitem}
\usepackage{stix}
\usepackage{graphicx,xtab}
\usepackage{makecell}
\usepackage{multirow}
\usepackage{tabularx}
\usepackage[top=30mm,bottom=30mm,left=25mm,right=25mm]{geometry}

\usepackage[colorinlistoftodos,  textsize=tiny]{todonotes} 					 
\presetkeys{todonotes}{fancyline, color=blue!30}{}

\usepackage[bb=boondox]{mathalfa}
\usepackage{graphicx}
\usepackage[colorlinks=true,citecolor=red,urlcolor=blue,linkcolor=red,bookmarksopen=true,unicode=true,pdffitwindow=false]{hyperref}
\usepackage[english]{babel}
\usepackage[languagenames,fixlanguage]{babelbib}
\hypersetup{pdfauthor={}}
\hypersetup{pdftitle={Ergodic convergence rates for storage processes}}

\theoremstyle{plain}
\newtheorem{theorem}{Theorem}[section]

\newtheorem{lemma}[theorem]{Lemma}
\newtheorem{proposition}[theorem]{Proposition}
\theoremstyle{definition}

\newtheorem*{custom}{\customname}

\newcommand{\customname}{Theorem}
\newcommand {\Prob} {\ensuremath{\mathbb{P}}}
\newcommand {\R} {\ensuremath{\mathbb{R}}}
\newcommand {\ZZ} {\ensuremath{\mathbb{Z}}}
\newcommand {\N} {\ensuremath{\mathbb{N}}}

\newcommand{\process}[1]{\{#1(t)\}_{t\geq0}}

\newcommand {\D} {{\rm d}}
\newcommand {\V} {\ensuremath{\mathcal{V}}}
\newcommand {\W} {\ensuremath{\mathcal{W}}}
\newcommand {\Ind} {\ensuremath{\mathbb{1}}}

\newcommand{\df}{\coloneqq}

\newcommand{\E}{\mathrm{e}}

\numberwithin{equation}{section}

\title{Stability of storage processes with general release rates}

\author[M.\ Bre\v{s}ar]{Miha Bre\v{s}ar}
\address[Miha Bre\v{s}ar]{School of Data Science\\The Chinese University of Hong Kong\\ Shenzhen\\China}
\email{mihabresar@cuhk.edu.cn}

\author[A.\ Mijatovi\'c]{Aleksandar Mijatovi\'c}
\address[Aleksandar Mijatovi\'c]{Department of Statistics\\University of Warwick\\ Coventry\\UK}
\email{a.mijatovic@warwick.ac.uk}

\author[N.\ Sandri\'{c}]{Nikola Sandri\'{c}}
\address[Nikola\ Sandri\'{c}]{Department of Mathematics\\University of Zagreb\\ Zagreb\\Croatia}
\email{nikola.sandric@math.hr}

\makeatletter
\@namedef{subjclassname@2020}{%
	\textup{2020} Mathematics Subject Classification}
\makeatother

\subjclass[2020]{60K25 (Primary) 
60J25, 60J76 (Secondary)}
\keywords{Storage process, stationary law,  convergence rates, total variation, Wasserstein distance}

\begin{document}
\allowdisplaybreaks[4]

\begin{abstract}
This paper quantifies the ergodicity and the rate of decay of the tail of the stationary distribution for a broad class of storage models, encompassing constant, linear, and power-type release rates with both finite and infinite activity input process. Our results are expressed in terms of the asymptotics of the release rate, the tail-decay rate  of the L\'evy measure of the input process and its (possibly infinite) first moment. Our framework unifies and significantly extends classical results on the stability of storage models. Under certain regularity assumptions, 
we also provide upper bounds on the stability in the Wasserstein distance.
 \end{abstract}

\maketitle

\section{Introduction}\label{s1}

Storage processes are classical and well-studied models in applied probability. Such processes are typically defined by an input process given by a subordinator (i.e., a L\'evy process with non-decreasing trajectories) and a release rate function. Despite their rich history see, for instance,~\cite{Brockwell-1977,Brockwell-Resnick-Tweedie-1982,Meyn-Tweedie-AdvAP-III-1993,Tuominen-Tweedie-1979,Yamazato-2005,MR1978607}, and deep connections to modern literature~\cite{MR4781632,MR3595769,MR3975897,MR4370826,MR4331112,MR4940338,MR3379923,MR4243160,MR4947855}, quantitative results on their ergodic behaviour remain scarce and are largely confined to specific or simplified cases (e.g., exponential ergodicity for compound Poisson input process~\cite{Meyn-Tweedie-AdvAP-III-1993}).

The primary obstacle lies in the inherent complexity of these models: the non-local (jump) nature of the subordinator, combined with the potential space-inhomogeneity of the release rate, renders classical analytical techniques inadequate, particularly in the case when release rate  is non-affine (cf.~\cite{MR4370826}). In this paper, we overcome these challenges by employing modern probabilistic methods~\cite{Bresar-Mijatovic-2024} to \emph{quantify} the stability of general storage models. Our results significantly extend the classical theory~\cite{Brockwell-1977,Meyn-Tweedie-AdvAP-III-1993,Yamazato-2005} both in terms of the strength of the convergence guarantees and the breadth of models they encompass.
Perhaps surprisingly, despite their apparent simplicity, storage models exhibit a remarkably rich and diverse range of phenomena. As illustrated in Table~\ref{tab:main} below, our results show that  these models exhibit behaviours ranging from polynomial to uniform ergodicity. Crucially, the asymptotic properties of storage models can be deduced directly from the model parameters.

The storage process $\{X(t)\}_{t\ge0}$ studied in this paper takes the following form:
\begin{equation} \label{eq1_intro} X(t)=x+A(t)-\int_0^tr(X(s))\D s. \end{equation} 
 Here $x\in[0,\infty)$ denotes the initial content of the storage system, $\{A(t)\}_{t\ge0}$ is a non-decreasing  pure-jump L\'evy process modelling the cumulative input, and $u\mapsto r(u)$ (for $u\in[0,\infty)$) is the positive release rate when the content satisfies $u>0$.    The main results concerning convergence in total variation distance are summarized in Table~\ref{tab:main} below, with proofs provided in Section~\ref{s2}. Before presenting these results, we collect some notation that will be used in the remainder of the paper.  Recall that the Laplace transform of $\{A(t)\}_{t\ge0}$ is given by $$\mathbb{E}\left[\E^{-\lambda A(t)}\right]=\exp\left\{t\int_{[0,\infty)}(\E^{-\lambda u}-1)\upnu(\D u)\right\},\qquad \lambda\ge0,$$ where $\upnu(\D u)$ is the corresponding L\'evy measure satisfying $\int_{(0,\infty)} (u\wedge 1)\upnu(\D u) <\infty.$
For functions $f,g:(0,\infty)\to(0,\infty)$, we write $f(t)\preccurlyeq g(t)$ if $\limsup_{t\to\infty}f(t)/g(t)<\infty.$ Moreover, for $\alpha \in \R\setminus \{0\}$, we write  $f(t)\approx t^\alpha$ if for every $\epsilon \in (0,|\alpha|)$, we have
$$t^{\alpha-\epsilon}\preccurlyeq f(t) \preccurlyeq t^{\alpha+\epsilon}.$$
For a $\sigma$-finite measure $\upmu$ on the Borel $\sigma$-algebra $\mathfrak{B}((0,\infty))$, finite on the complement of any neighbourhood of zero, define its  first moment $m_\upmu$ and  tail function $\bar\upmu$ are defined as follows 
\begin{equation}
\label{eq:nu_first_moment}
m_\upmu\df\int_{(0,\infty)}u\upmu(\D u)\in[0,\infty]\quad\text{and}\quad \bar\upmu(u)\df\upmu((u,\infty)),\quad u>0.
\end{equation}
The storage process in~\eqref{eq1_intro} exhibits a wide range of behaviours. We begin by classifying transience and null/positive-recurrent cases of $\{X(t)\}_{t\ge0}$ (see Theorem~\ref{prop:transience} below for more details).

\begin{table}[!h]
\label{tab:1}
\centering
\small
\setlength{\tabcolsep}{5pt}
\renewcommand{\arraystretch}{1.5}
\begin{tabularx}{\textwidth}{|c|X|X|}
\hline
\textbf{Regime} & \textbf{Conditions} & \textbf{Notes/Examples} \\
\hline
\makecell{Transient} & 
\makecell[l]{(a) $m_\upnu < \infty$, $\limsup r(u) < m_\upnu$ \\[0.5ex] 
(b) $m_\upnu = \infty$, $\liminf \frac{u}{r(u)}\int_0^\infty \frac{\bar\upnu(uv)}{1+v^2}\,dv > 1$} & 
\makecell[l]{Ex: $\bar\upnu(u) \approx u^{-\alpha}$, $r(u) \approx u^\beta$: \\ 
transience if $\alpha + \beta < 1$} \\
\hline
Null recurrent & $m_\upnu < \infty$, $r(u) = m_\upnu$ for $u \geq u_0$ & Mean zero increments when $X(t)$  is large. \\
\hline
Pos. recurrent & $\limsup_{u\to\infty} \int_0^\infty \frac{\bar\upnu(v)}{r(u+v)}\,dv < 1$ & Ex: $\alpha + \beta > 1$, $\alpha,\beta \in (0,1)$ \\
\hline
\end{tabularx}
\caption{Asymptotic regimes based on the input ($\bar\upnu$) and release ($r$) functions.}
\end{table}
Our main results, stated in Theorems~\ref{tm:TV} and~\ref{prop:upper_bounds_super_linear} below, concern the quantification of the asymptotic behaviour in the positive recurrent regime. In this case there exists the invariant distribution $\uppi$ on 
$\mathfrak{B}([0,\infty))$
 of the process $\{X(t)\}_{t\ge0}$  (see~\eqref{eq:invariant} below for details).
 Table~\ref{tab:main} summarises the findings in Theorems~\ref{tm:TV} and~\ref{prop:upper_bounds_super_linear} for important special cases
 of constant, linear, and power-type release rates, highlighting the range of phenomena exhibited by the storage models.
\begin{table}[h]
\centering
\small
\setlength{\tabcolsep}{3pt}
\renewcommand{\arraystretch}{1.1}
\begin{tabularx}{\textwidth}{|X|X|c|c|}
\hline
\textbf{Release rate $r(u)$ as $u\uparrow\infty$} & \textbf{Input Rate $\bar\upnu(u)$ as $u\uparrow\infty$} & \textbf{Conv. Rate in TV as $t\uparrow\infty$} & \textbf{$\bar\uppi(u)$ as $u\uparrow \infty$} \\
\hline
$r(u)=a$ (large $u$) & $\bar \upnu(u)\preccurlyeq\exp(-cv)$ & Exponential & Exponential \\
\cline{2-4}
$m_\upnu < a$ & $\bar \upnu(u) \approx v^{-\alpha}$, $\alpha>1$ & $\approx t^{1-\alpha}$ & $\approx u^{1-\alpha}$ \\
\hline
$r(u)\approx u$ & $\bar \upnu(u)\approx v^{-\alpha}$, $\alpha>0$ & Exponential & $\approx u^{-\alpha}$ \\
\cline{2-4}
& $\bar \upnu(u)\approx (\ln v)^{-\alpha}$, $\alpha>1$ & $\approx t^{1-\alpha}$ &  $\approx (\ln u )^{1-\alpha}$\\
\hline
$r(u)\approx u^\beta$, $\beta\in(0,1)$ & $\bar\upnu(u)\preccurlyeq \exp(-u^{1-\beta}) $ & Exponential & $\preccurlyeq \exp(-u^{1-\beta})$ \\
\cline{2-4}
& $\bar \upnu(u)\approx u^{-\alpha}$, $\alpha\!+\!\beta\!>\!1$ & $\approx t^{\frac{1-\alpha-\beta}{1-\beta}}$ & $\approx u^{1-\alpha-\beta}$ \\
\hline
$r(u)\approx u^\beta$, $\beta\in(1,\infty)$ & $\bar \upnu(u)\preccurlyeq u^{-\alpha}$ $\alpha>0$ & Uniform & $\preccurlyeq u^{1-\alpha-\beta}$ \\
\hline
\end{tabularx}
\caption{A summary of ergodic properties for positive recurrent storage models. Here $\uppi$ denotes the invariant measure for $\{X(t)\}_{t\ge0}$  
with $\bar\uppi(u)$ the corresponding tail function. The ``Convergence Rate'' column quantifies the rate of decay of total variation $\|\mathbb{P}^x(X(t)\in\cdot) - \uppi\|_{\mathrm{TV}}$, as $t\to\infty$. Note that the value of the first moment $m_\upnu$ of the L\'evy measure  is only relevant for the asymptotic behaviour of the storage model $\{X(t)\}_{t\ge0}$   in the case where $r(u)$ is asymptotically bounded.}
\label{tab:main}
\end{table}

A short \href{https://youtu.be/TnvkYk2GxnM?si=1YFwslWk06w-opzD}{YouTube video} discussing the results in Table~\ref{tab:main} and the ideas behind their proofs is available~\cite{YouTube_talk}.
To the best of our knowledge, Table~\ref{tab:main} provides first results on the subgeometric ergodicity for storage processes.
Related results appear in classical literature, e.g., ~\cite[Lemma 5.2]{Brockwell-1977},~\cite[Proposition 11]{Brockwell-Resnick-Tweedie-1982},~\cite[Sections 7 and 8]{Harrison-Resnick-1976},~\cite[Theorem 9.1]{Meyn-Tweedie-AdvAP-III-1993},~\cite[Theorem 3]{Tuominen-Tweedie-1979}, \cite{Yamazato-2000}, and~\cite{Yamazato-2005}, where ergodicity conditions (without explicit convergence rates) were studied under the finite-activity assumption of the input process (i.e., $\bar\upnu(0) < \infty$).

In modern literature, storage models are often analysed in the special case of shot-noise processes (i.e., the case of a linear release rate) see, e.g.,~\cite[Sec.~3]{MR4331112} and the references therein. A more general class of such models, given by affine storage processes with two-sided jumps, which  encompasses shot-noise processes as a special case (e.g., when the constant premium rate is zero), is also widely studied~\cite{MR4370826}.  These models are particularly important in queueing theory~\cite{MR4951813} and include cases with both infinite- ($\bar\upnu(0) = \infty$)~\cite{MR3975897} and finite-activity subordinators ($\bar\upnu(0) < \infty$)~\cite{MR4331112}. By including general sublinear/superlinear release rates that generalise the affine cases, such as those studied in~\cite{MR4370826}, we encounter both polynomial and uniform convergence rates, significantly extending the known range of phenomena. Furthermore, our models include finite and infinite activity subordinators; however, the results demonstrate that this distinction crucially does not play a role in asymptotic behaviour, which is determined by the tail of the L\'evy measure and its first moment. We expect that the methods applied in the present paper would also characterise the stability of the general release rate storage processes with two-sided jumps. The details are left for future research. 

A condition for exponential ergodicity of $\{X(t)\}_{t\ge0}$ was established in~\cite[Theorem 9.1]{Meyn-Tweedie-AdvAP-III-1993}, again under the assumption $\bar\upnu(0) < \infty$. Our results in Section~\ref{ss3} extend known results on ergodicity, by allowing both finite and infinite activity inputs, establishing the actual convergence rates in subexponential cases (not only upper bounds on convergence rates) in terms of the model parameters.

Finally we note that the ergodicity of general Markov processes   in the total variation distance involve
understanding the irreducibility and aperiodicity of the process (see Section~\ref{s22} for details). When the process is not sufficiently regular (i.e., either reducible or periodic), the topology induced by the total variation distance becomes too ``fine'', precluding convergence.  Thus,  the convergence towards an invariant probability measure  (if it exists) may fail in the total variation  metric and can only be established in a weaker sense. In such cases, it is natural to resort to  Wasserstein distances, which, in a certain sense, induce a ``rougher'' topology. 
In Theorem~\ref{eqWASS1} below we give general conditions for the convergence of $\{X(t)\}_{t\ge0}$ in the Wasserstein metric to its stationary measure, which produces both exponential and polynomial upper bounds. The lower bounds in the polynomial case are left as an open problem.

The remainder of the article is organized as follows. 
Section~\ref{s22} is devoted to structural properties of storage processes, 
as well as to a review of definitions and general results on the ergodic theory 
of Markov processes with respect to the total variation and Wasserstein distances. 
In Section~\ref{s22}, we also formulate general versions of our main theorems, 
from which the results in Tables~\ref{tab:1} and~\ref{tab:main} are corollaries. We conclude Section~\ref{s22}
with a review of related literature. 
Finally, Section~\ref{s2} contains the proofs of the main results.

\section{General results}\label{s22}
In this section, we formulate and prove the general versions of the main results, which in turn yield the results of the previous section. We begin by describing some structural properties of storage processes.

\subsection{Structural properties of the model}\label{ss1}  Let $\{A(t)\}_{t\ge0}$
be defined on a stochastic basis $(\Omega,\mathcal{F},$ $\{\mathcal{F}_t\}_{t\ge0},\mathbb{P})$ satisfying the usual conditions. Assume that  the release rate $r(u)$ satisfies the following conditions:
\begin{description}
	\item[(C1)] $r:[0,\infty)\to[0,\infty)$ is strictly positive  if, and only if, $u>0$.
	\item[(C2)] For any $\rho>0$, there exists $\Gamma_\rho>0$ such that for all $0\le u, v\le \rho$,  $$|r(u)-r(v)|\le \Gamma_\rho|u-v|.$$
\end{description}

\noindent Under $\textbf{(C1)}$, it has been shown in   \cite[Propositions 1-3]{Brockwell-Resnick-Tweedie-1982}  that for any $x\ge0$ the equation 
\begin{equation} \label{eq1} X(x,t)=x+A(t)-\int_0^tr(X(x,s))\D s. \end{equation}
admits a  strong nonexplosive  solution $\{X(x,t)\}_{t\ge0}$ which is a nonnegative time-homogeneous strong Markov process (with respect to $(\Omega,\mathcal{F},\{\mathcal{F}_t\}_{t\ge0},\mathbb{P})$) with c\`{a}dl\`{a}g sample paths and  transition kernel $p(t,x,\D y):=\Prob(X(x,t)\in \D y)$, $t\ge0$, $x\ge0$. If, in addition,  $\textbf{(C2)}$ holds, then $\{X(x,t)\}_{t\ge0}$ is the unique solution to \eqref{eq1}  (see \cite[Theorem 3.1]{Albeverio-Brzezniak-Wu-2010}). In the context of Markov processes, it is natural for the underlying probability measure to depend on the initial position of the process. Using standard arguments (Kolmogorov extension theorem), it is well known that for each $x\ge0$, the transition kernel defined above defines  a unique probability measure $\mathbb{P}^x$ on the canonical (sample-path)  space such that the projection process, denoted by $\process{X}$, is a strong Markov process (with respect to the completion of the corresponding natural filtration). The process $\process{X}$ has c\`{a}dl\`{a}g   sample paths
and  the same finite-dimensional distributions (with respect to $\mathbb{P}^x$) as $\{X(x,t)\}_{t\ge0}$ (with respect to $\mathbb{P}$). Since we are interested only in distributional properties of the solution to   \eqref{eq1}, we will henceforth deal with $\process{X}$ rather than  $\{X(x,t)\}_{t\ge0}$. 
 Further, recall that the Laplace transform of $\{A(t)\}_{t\ge0}$ is given by $$\mathbb{E}\left[\E^{-\lambda A(t)}\right]=\exp\left\{t\int_{[0,\infty)}(\E^{-\lambda u}-1)\upnu(\D u)\right\},\qquad \lambda\ge0,$$ where $\upnu$ is the corresponding L\'evy measure which satisfies $\int_{[0,\infty)} (u\wedge 1)\upnu(\D u) <\infty.$  According to \cite{Albeverio-Brzezniak-Wu-2010} and \cite{Masuda-2007},
for any
$f\in \mathcal{C}^1([0,\infty))$ such that $u\mapsto\int_{(1,\infty)}f(u+v)\upnu(\D v)$ is locally bounded, the following holds:
$$f(X(t))-f(X(0))-\int_0^t\mathcal{L}f(X(0))\D s,\qquad t\geq0,$$ is a $\mathbb{P}^x$-local martingale for every $x\ge0$, where $$\mathcal{L}f(u):=- r(u)f'(u)+\int_{[0,\infty)}\left(f(u+v)-f(u)\right)\upnu(\D  v).$$  
We remark that when $\upnu((0,\infty))<\infty$, i.e., $\{A(t)\}_{t\ge0}$ is a compound Poisson process, it has been shown in \cite{Harrison-Resnick-1976}  that the above properties hold if we replace $\textbf{(C1)}-\textbf{(C2)}$ with the following conditions:
\begin{description}
	\item[(\={C}1)] $r:[0,\infty)\to[0,\infty)$ is   left continuous and has a strictly positive right limit at every point in $(0,\infty)$ with  $r(u)=0$ if, and only if, $u=0$.
	\item[(\={C}2)] $\displaystyle\int_0^u(r(v))^{-1}\D v<\infty$ for all $u>0$.
\end{description}

\subsection{Ergodicity of Markov processes} \label{ss2}

We now recall some definitions and general results from the ergodic theory of Markov processes. Our main references are \cite{Meyn-Tweedie-AdvAP-II-1993} and \cite{Tweedie-1994}.
Let $(\Omega,\mathcal{F}, \process{\mathcal{F}}, \process{\theta},\process{M},\{\Prob^x\}_{x\in S})$, denoted by  $\process{M}$ in the sequel, be a Markov process with c\`adl\`ag sample paths and  state space
$(S,\mathfrak{B}(S))$, where $S$ is a locally compact and separable metric space, and $\mathfrak{B}(S)$ is the corresponding Borel $\sigma$-algebra (see \cite{Blumenthal-Getoor-Book-1968}).  We let $p(t,x,\D y):=\Prob^x(M_t\in \D y)$, $t\ge0$, $x\in S$,  denote the corresponding transition kernel.  
Also,  assume that $p(t,x,\D y)$
is a probability measure, i.e., $\process{M}$ does not admit a cemetery point
in the sense of \cite{Blumenthal-Getoor-Book-1968}. Observe that this is not a restriction since,
as we have already commented, $\process{X}$ is nonexplosive.
The process $\process{M}$ is called
\begin{enumerate}
	\item [(i)]
	$\upphi$-irreducible if there exists a $\sigma$-finite measure $\upphi$ on
	$\mathfrak{B}(S)$ such that whenever $\upphi(B)>0$, we have
	$\int_0^{\infty}p(t,x,B)\D t>0$ for all $x\in S$.\end{enumerate}
Let us remark that if $\{M_t\}_{t\ge0}$ is  $\upphi$-irreducible, then the irreducibility measure $\upphi$ can be
maximized. This means that there exists a unique ``maximal'' irreducibility
measure $\uppsi$ such that for any measure $\bar{\upphi}$,
$\{M_t\}_{t\ge0}$ is $\bar{\upphi}$-irreducible if, and only if,
$\bar{\upphi}$ is absolutely continuous with respect to $\uppsi$ (see \cite[Theorem~2.1]{Tweedie-1994}).
In view to this, when we refer to an irreducibility
measure we actually refer to the maximal irreducibility measure. The process $\process{M}$ is called
\begin{enumerate}
\item [(ii)]
transient if it is $\uppsi$-irreducible, and if there exists a countable
covering of $S$ with sets
$\{B_j\}_{j\in\N}\subseteq\mathfrak{B}(S)$, and for each
$j\in\N$ there exists a finite constant $c_j\ge0$ such that
$\int_0^{\infty}p(t,x,B_j)\D{t}\le c_j$ holds for all $x\in S$.
\item [(iii)]
recurrent if it is $\uppsi$-irreducible, and $\uppsi(B)>0$ implies
$\int_{0}^{\infty}p(t,x,B)\D{t}=\infty$ for all $x\in S$.
\end{enumerate}
It is well known that every $\uppsi$-irreducible Markov
process is either transient or recurrent (see \cite[Theorem
2.3]{Tweedie-1994}).

A
(not necessarily finite) measure $\uppi$ on $\mathfrak{B}(S)$ is called invariant for
$\process{M}$ if
\begin{equation}
\label{eq:invariant}
\int_{S}p(t,x,\D y)\uppi(\D x)=\uppi(\D y),\qquad t\ge0.
\end{equation}
It is well known that if $\process{M}$ is
recurrent, then it possesses a unique (up to constant
multiples) invariant measure 
(see \cite[Theorem~2.6]{Tweedie-1994}).
If the 
invariant measure is
finite, then it may be normalized to a probability measure. If
$\process{M}$ is recurrent with finite invariant measure, then $\process{M}$ is called 
positive recurrent; otherwise it is called null recurrent. Note that a transient 
Markov process cannot have a finite invariant measure. 
Indeed, assume that
$\process{M}$ is transient and that it admits a
finite invariant measure $\uppi$. Fix some $t>0$. Then, for each $j\in\N$, with $c_j$ and $B_j$ as in (ii) above, we have
\begin{equation*}
t\uppi(B_j) =
\int_0^{t}\int_{S}p(t,x,B_j)\uppi(\D x)\D s
\le c_j\uppi(S).
\end{equation*}
Now, by
letting $t\to\infty$ we obtain $\uppi(B_j)=0$ for all
$j\in\N$, which is impossible. The process $\process{M}$ is called 
\begin{enumerate}
	\item [(iv)] ergodic
if it possesses an invariant probability 
measure $\uppi$ and there exists a nondecreasing function
$\rho:[0,\infty)\to[1,\infty)$ such that 
\begin{equation*}
\lim_{t\to\infty}\rho(t)\lVert p(t,x,\D {y})
-\uppi(\D {y})\rVert_{{\rm TV}} =0,\qquad x \in S.
\end{equation*} 
	\item [(v)] uniformly ergodic 
if it possesses an invariant probability 
measure $\uppi$  such that 
\begin{equation*}
\lim_{t\to\infty}\sup_{x\in S}\lVert p(t,x,\D {y})
-\uppi(\D {y})\rVert_{{\rm TV}} =0.
\end{equation*}
\end{enumerate}
Here, $\lVert\upmu\rVert_{{\rm TV}}:=\sup_{B\in\mathfrak{B}(S)}|\upmu(B)|$ is the total variation distance  of a signed measure $\upmu$ (on $\mathfrak{B}(S$). Observe that ergodicity implies $\uppi$-irreducibility and positive recurrence.
We say that $\process{M}$ is subgeometrically ergodic if it is ergodic and 
$\lim_{t\to\infty}\ln \rho(t)/t=0$, and 
that it is geometrically ergodic if it is ergodic and
$\rho(t)=\E^{\kappa t}$ for some $\kappa>0$. On the other hand, the  rate of convergence  in (v)  is necessarily geometric. Namely,  from (v) we have that for any $t_0>0$,
\begin{equation*}
\lim_{n\to\infty}\sup_{x\in S}\lVert p(nt_0,x,\D {y})
-\uppi(\D {y})\rVert_{{\rm TV}} =0,
\end{equation*} i.e., the Markov chain $\{M(nt_0)\}_{n\in\ZZ_+}$ is uniformly ergodic. Consequently, from \cite[Theorem 16.0.2]{Meyn-Tweedie-Book-2009} it follows that there exist $B,\beta>0$ such that 
\begin{equation*}
\lVert p(nt_0,x,\D {y})
-\uppi(\D {y})\rVert_{{\rm TV}} \le B\E^{-\beta  n t_0},\qquad  x \in S, \quad n\in\ZZ_+.\end{equation*} Since every $t\ge0$ can be clearly represented as $t=nt_0+s$ for some $n\in\ZZ_+$ and $s\in[0,t_0)$, we then conclude \begin{align*}\lVert p(t,x,\D {y})
-\uppi(\D {y})\rVert_{{\rm TV}}&=\lVert p(nt_0+s,x,\D {y})
-\uppi(\D {y})\rVert_{{\rm TV}}\\&\le\lVert p(nt_0,x,\D {y})
-\uppi(\D {y})\rVert_{{\rm TV}}  \\&\le B\E^{-\beta  n t_0}
\\&=B\E^{\beta s}\E^{-\beta  t}\\&\le B\E^{\beta t_0}\E^{-\beta  t},\qquad x \in S, \quad t>0.\end{align*}
As in the classical discrete setting, crucial steps in studying the ergodicity of general Markov processes  involve
understanding the periodic behavior of the process and  the structure  of  the corresponding ``singletons''. We adopt the definition of aperiodicity from  \cite{Meyn-Tweedie-AdvAP-II-1993}. 
The process $\process{M}$ is said to be 
\begin{enumerate}
	\item [(vi)] aperiodic if it admits an irreducible skeleton chain, i.e., 
	there exist $t_0>0$ and a $\sigma$-finite measure $\upphi$ on
	$\mathcal{B}(S)$, such that $\upphi(B)>0$ implies
	$\sum_{n=0}^{\infty} p(nt_0,x,B) >0$ for all $x\in S$.
\end{enumerate}
A set $C\in\mathfrak{B}(S)$ is said to be   
\begin{enumerate}
	\item [(vii)] petite if  there exist a probability   measure $\upmu_C$ on $\mathfrak{B}((0,\infty))$ and a non-trivial measure $\upeta_C$ on $\mathfrak{B}(S)$, such that $\int_0^\infty p(t,x,B)\upmu_C(\D t)\ge\upeta_C(B)$ for all $x\in C$ and $B\in\mathfrak{B}(S)$.
\end{enumerate} Recall that  $\uppsi$-irreducibility implies that the state space (in this case $(S,\mathfrak{B}(S)$) can be covered by a countable union of petite sets (see \cite[Propositio 4.1]{Meyn-Tweedie-AdvAP-II-1993}.   
Intuitively,  petite sets play the role of singletons for Markov processes on general state spaces (see \cite[Section 4]{Meyn-Tweedie-AdvAP-II-1993} and \cite[Chapter 5]{Meyn-Tweedie-Book-2009} for details).

We now recall the notion and some general facts about Wasserstein distances. Let $\varrho$ be a metric on $S$. 
For $p\ge0$, let $\mathcal{P}_{p}$ denote the space of all probability measures $\upmu$ on $\mathfrak{B}(S)$ having finite $p$-th moment, i.e., $$\int_{S}\varrho(x_0,x)^p\upmu(\D x)<\infty$$ for some (and hence  any) $x_0\in S$. We also denote $\mathcal{P}_{0}$  by $\mathcal{P}.$ 
For $p\ge1$ and  $\upmu,\upnu\in\mathcal{P}$, the $\mathcal{L}^p$-Wasserstein distance between $\upmu$ and $\upeta$ is defined as $$\W_{p}(\upmu,\upeta):=\inf_{\Pi\in\mathcal{C}(\upmu,\upeta)}\left(\int_{S\times S}\varrho(x,y)^p\Pi(\D x,\D y)\right)^{1/p},$$ where  $\mathcal{C}(\upmu,\upeta)$ is the set of couplings of $\upmu$ and $\upeta$, i.e., $\Pi\in\mathcal{C}(\upmu,\upeta)$ if, and only if, $\Pi$ is a probability measure on $S\times S$ with $\upmu$ and $\upeta$ as its marginals. It is well known that $\W_{p}$ satisfies the axioms of a (not necessarily finite) distance on $\mathcal{P}$. The restriction  of $\W_{p}$ to  $\mathcal{P}_{p}$   defines a finite distance.
Since  $S$ is a locally compact and separable metric space,  it is well known that $(\mathcal{P}_{p},\W_{p})$ is also a  locally compact and separable metric space (see \cite[Theorem 6.18]{Villani-Book-2009}). Of  particular interest to us is the case when $S=[0,\infty)$ and $\varrho(x,y)=|x-y|$. 
For more on Wasserstein distances, we refer the readers to \cite{Villani-Book-2009}.

\subsection{General results} \label{ss3}
In this subsection, we present conditions for transience, null recurrence, positive recurrence, as well as  (sub)geometric ergodicity with respect to the total variation distance and/or a class of Wasserstein distances of the process $\process{X}$. 
 
\begin{theorem}
	\label{prop:transience} Assume $\textbf{(C1)-(C2)}$ (or $\textbf{(\={C}1)-(\={C}2)}$ in the case when $\bar\upnu(0)<\infty$), that $\process{X}$ is  $\uppsi$-irreducible and aperiodic, and that every compact set is petite for $\process{X}$. 
	\begin{itemize} 
		\item[(i)] The process $\process{X}$ is transient if either of the two conditions hold:
		\begin{enumerate}
			\item[(a)]  $\displaystyle\limsup_{u\to\infty} r(u)<m_\upnu$.
			\item[(b)]  $\displaystyle\liminf_{u\to\infty} \frac{u}{r(u)}\int_0^\infty \frac{1}{1+v^2}\bar\upnu(uv)\D v > 1$.
		\end{enumerate}
	\item[(ii)] If, for some $u_0\in(0,\infty)$, $r(u) = m_\upnu$ for all $u\geq u_0$, the process  $\process{X}$ is null recurrent.
			\end{itemize}
	\end{theorem}
\noindent Condition~(a) in Theorem~\ref{prop:transience}~(i) yields transience in cases where the drift is bounded, e.g., $\lim_{u\to\infty}r(u) =0$ leads to a transient process for any choice of cumulated inputs $\{A(t)\}_{t\ge0}$.  Condition~(b) is relevant when $\lim_{u\to\infty}r(u)=\infty$ and  $m_\upnu=\infty$. Examples include cases where $\bar\upnu(u)= au^{-\alpha}$ and $r(u) = b u^\beta$ (with $a,b,\alpha,\beta>0$), in which case condition~(b) holds if $\alpha + \beta<1$. Further results for the special case of non-decreasing $r(u)$ are given in \cite{Yamazato-2000} and \cite{Yamazato-2005}.

Let $\varphi:[1,\infty)\to(0,\infty)$ be a  nondecreasing, differentiable, and concave  function.
On $[1,\infty)$, define $$\Phi(t):=\int_1^{t}\frac{\D s}{\varphi(s)}\qquad \text{and}\qquad \bar{\V}(u):=\Phi^{-1}\left(\int_1^u\frac{\D v}{r(v)}+1\right).$$   Observe that $\bar{\V}(u)$ is continuous on $[1,\infty)$, continuously differentiable on $(1,\infty)$, and $\bar{\V}(1)>1$. Let $\V:[0,\infty)\to[1,\infty)$ be continuously differentiable, and such that $\V(u)=\bar{\V}(u)$ for $u\ge1$. Clearly,
$$\V'(u)=\frac{\varphi(\V(u))}{r(u)},\qquad u\ge1.$$
Assume that
\begin{description}
	\item[(C3)] $\displaystyle\int_{[1,\infty)}\V(u+v)\upnu(\D v)<\infty$ for all $u\ge0$.
\end{description}
Under assumption \textbf{(C3)}, we have 
$$\mathcal{L}\V(u)=-\varphi(\V(u))+\int_{[0,\infty)}\left(\V(u+v)-\V(u)\right)\upnu(\D  v).$$  
By applying Fubini's theorem, we get
$$\int_{[0,\infty)}\left(\V(u+v)-\V(u)\right)\upnu(\D  v)=\int_{[0,\infty)}\int_{u}^{u+v}\V'(w)\D w\upnu(\D  v)=\int_0^\infty\V'(u+v)\bar\upnu(v)\D v.$$ Thus, we conclude that
$$\mathcal{L}\V(u)=-\varphi(\V(u))+\int_0^\infty\V'(u+v)\bar\upnu(v)\D v.$$

\begin{theorem}\label{tm:TV} 
	\begin{itemize}
		\item [(i)] Assume $\textbf{(C1)-(C3)}$ (or $\textbf{(\={C}1)-(\={C}2)}$ and  $\textbf{(C3)}$ in the case when $\bar\upnu(0)<\infty$), that $\process{X}$ is  $\uppsi$-irreducible and aperiodic, and that every compact set is petite for $\process{X}$.  If
	\begin{equation}\label{eq2}\limsup_{u\to\infty}\frac{\int_0^\infty\frac{\varphi(\V(u+v))}{r(u+v)}\bar\upnu(v)\D v}{\varphi(\V(u))}<1,\end{equation}
	then $\process{X}$ admits a unique invariant measure $\uppi\in\mathcal{P}$ satisfying 
	$$\lim_{t\to\infty}\varphi(\Phi^{-1}(t))\lVert p(t,x,\D y)-\uppi(\D y)\rVert_{{\rm TV}}=0,\qquad x\ge0.$$ Moreover, the invariant measure $\uppi$ satisfies \begin{equation}\label{tail}\limsup_{u\to\infty}(\varphi(\Phi^{-1}(u))\vee\varphi(\V(u)) )\bar\uppi(u)<\infty.\end{equation}
	\item[(ii)] Assume that $\bar\upnu(u)>0$ for all $u\in(0,\infty)$   and $u\mapsto 1/\bar\upnu(u)$ is differentiable, increasing and submultiplicative\footnote{A function $f:(0,\infty)\to(0,\infty)$ is submultiplicative if $f(u+v)\le cf(u)f(v)$ for some $c>0$ and all $u,v\in(0,\infty)$.}. Assume also that  $u\mapsto r(u)/u$ is decreasing with $\lim_{u\to\infty} r(u)/u = 0$. Then, if $\process{X}$ is ergodic with an invariant measure $\uppi\in\mathcal{P}$, for any $\epsilon>0$ there exists $c_\epsilon>0$ satisfying 
	$$
	\bar\uppi(u)\ge c_\epsilon \frac{\bar\upnu(u)}{r(u)}u^{1-\varepsilon}\qquad u\ge1.
	$$
Denote $L_\epsilon(u)\df c_\epsilon u^{1-\varepsilon}\bar\upnu(u)/r(u)$. If there exists increasing $h:[1,\infty)\to[1,\infty)$ such that $\mathbb{E}^x [h(X(t))]\leq h(x) + ct$ (for some $c\in(0,\infty)$ and all $x\ge0$ and $t\geq 1$) and $F(u)\coloneqq uL_\epsilon(h^{-1}(u))$ is increasing with $\lim_{u\to\infty} F(u) = \infty$, then the following lower bound holds
	$$
	\| p(t,x,\D y)-\uppi(\D y)\|_{\mathrm{TV}}\ge \frac{1}{2}L_\epsilon\circ h^{-1}\circ F^{-1}(2(h(x) + ct)).
	$$
	\item[(iii)] Assume that $\bar\upnu(u)>0$ for all $u\in(0,\infty)$   and $u\mapsto 1/\bar\upnu(\exp(u))$ is differentiable, increasing and submultiplicative. Assume also that  $u\mapsto r(u)/(u\log u)$ is decreasing eventually with $\lim_{u\to\infty} r(u)/(u \log u) = 0$. Then, if $\process{X}$ is ergodic with an invariant measure $\uppi\in\mathcal{P}$, for any $\epsilon>0$ there exists $c_\epsilon>0$ satisfying 
	$$
	\bar\uppi(u)\ge c_\epsilon \frac{\bar\upnu(u)}{u^{\varepsilon}}\qquad u\ge1.
	$$
Denote $L_\epsilon(u)\df c_\epsilon u^{-\varepsilon}\bar\upnu(u)$. If there exists increasing $h:[1,\infty)\to[1,\infty)$ such that $\mathbb{E}^x [h(X(t))]\leq h(x) + ct$ (for some $c\in(0,\infty)$ and all $x\ge0$ and $t\geq 1$) and $F(u)\coloneqq uL_\epsilon(h^{-1}(u))$ is increasing with $\lim_{u\to\infty} F(u) = \infty$, then the following lower bound holds
	$$
	\| p(t,x,\D y)-\uppi(\D y)\|_{\mathrm{TV}}\ge \frac{1}{2}L_\epsilon\circ h^{-1}\circ F^{-1}(2(h(x) + ct)).
	$$
	\end{itemize}
\end{theorem}
If $\lim_{t\to\infty}\varphi'(t)=0$ in Theorem \ref{tm:TV}~(i), then $\process{X}$ is subgeometrically ergodic. Specifically, 
\begin{itemize}
	\item [(i)] If $\lim_{t\to\infty}\varphi(\Phi^{-1}(t))<\infty,$ then clearly $$\lim_{t\to\infty}\frac{\ln\left(\varphi(\Phi^{-1}(t))\right)}{t}=0.$$ 
	\item [(ii)] If $\lim_{t\to\infty}\varphi(\Phi^{-1}(t)=\infty,$ then by L'H\^{o}pital's rule:
	$$\lim_{t\to\infty}\frac{\ln\left(\varphi(\Phi^{-1}(t))\right)}{t}=\lim_{t\to\infty}\varphi'(\Phi^{-1}(t)).$$ However, since in this case necessarily $\lim_{t\to\infty}\Phi^{-1}(t)=\infty,$ we again conclude that the limit vanishes.
\end{itemize}
On the other hand, if $\liminf_{t\to\infty}\varphi'(t)>0$, since $\varphi(t)$ is nondecreasing, differentiable and concave (which, in particular, implies that $\varphi'(t)$ is also nondecreasing), there exist  $\Gamma\ge\gamma>0$ such that $$\gamma t-\gamma+\varphi(t)\le\varphi(t)\le \Gamma t-\Gamma+\varphi(1),\qquad t\ge1.$$
Thus, we have the following bounds $$\gamma\varphi(1)\E^{\gamma t}+\gamma^2-\gamma\varphi(1)-\gamma+\varphi(1)\le \varphi(\Phi^{-1}(t)\le  \Gamma\varphi(1)\E^{\Gamma t}+\Gamma^2-\Gamma\varphi(1)-\Gamma+\varphi(1),\qquad t\ge1.$$

By taking $\varphi\equiv1$, \eqref{eq2} reduces to: \begin{equation}\label{eq3}\limsup_{u\to\infty}\int_0^\infty\frac{\bar\upnu(v)}{r(u+v)}\D v<1,\end{equation} which is exactly the condition for ergodicity obtained in
\cite[Theorem 9.1]{Meyn-Tweedie-AdvAP-III-1993} (in the case when $\bar\upnu(0)<\infty$). 
The same condition was also derived  in \cite[Proposition 11]{Brockwell-Resnick-Tweedie-1982}, and stronger variants  (i.e., conditions implying \eqref{eq2}) can be found  in \cite[Lemma 5.2]{Brockwell-1977}, \cite[Sections 7 and 8]{Harrison-Resnick-1976}, \cite[Theorem 3]{Tuominen-Tweedie-1979}, \cite{Yamazato-2000} and \cite{Yamazato-2005}.
By taking $\varphi(u)=cu$ for some $c>0$, \eqref{eq2} reduces to: 
$$\limsup_{u\to\infty}\frac{\int_0^\infty\frac{\E^{c\int_1^{u+v}\frac{\D w}{r(w)}}}{r(u+v)}\bar\upnu(v)\D v}{\E^{c\int_1^u\frac{\D v}{r(v)}}}<1,$$ which is the condition for geometric ergodicity obtained in
\cite[Theorem 9.1]{Meyn-Tweedie-AdvAP-III-1993} (again, only  when $\bar\upnu(0)<\infty$).
Theorem \ref{tm:TV}~(i), generalizes these results by covering  subgeometric   ergodicity rates (e.g., the case when $\varphi(t)=t^\alpha$ for some $\alpha\in(0,1)$), and it does not require  linear form of the rate function $\varphi(t)$ in the geometric ergodicity case (e.g., $\varphi(t)=t+\ln(1+t)$). On the other hand, the results in Theorem~\ref{tm:TV}~(ii)—namely, the lower bounds on ergodic convergence rates for storage processes—are new and, to the best of our knowledge, have not previously been investigated in the literature. By establishing matching upper and lower bounds, we obtain a sharp quantitative characterization of the ergodic convergence rates.

The   proof of Theorem \ref{tm:TV}~(i) is based on the Foster-Lyapunov method for (sub)geometric ergodicity of Markov processes, as developed in \cite{Douc-Fort-Guilin-2009} and \cite{Down-Meyn-Tweedie-1995}.
The method consists of finding a petite set $C\in\mathfrak{B}([0,\infty))$ and constructing an appropriate  function $\V:[0,\infty)\to[1,\infty)$ (the so-called Lyapunov  function), which is contained in the domain of the extended generator $\mathcal{L}^{e}$  of the  process $\process{X}$ and which satisfies the   following Lyapunov equation: \begin{equation}
\label{eq:lyap}
\mathcal{L}^{e}\V(x)\le-\varphi(\V(x))+\chi\Ind_C(x),\qquad x\ge0,\end{equation} for some $\chi\in(0,\infty)$ (see \cite[Theorem 3.4]{Douc-Fort-Guilin-2009} and \cite[Theorem 5.2]{Down-Meyn-Tweedie-1995}). 
For  the class of differentiable functions  satisfying $\textbf{(C3)}$, $\mathcal{L}^{e}$ coincides with $\mathcal{L}$ (see \cite[Section 1]{Meyn-Tweedie-AdvAP-III-1993} for details). 
The equation in \eqref{eq:lyap} implies that for any $\delta>0$,
the $\varphi\circ\Phi^{-1}$-moment  of the $\delta$-shifted hitting time   $\tau_C^\delta:=\inf\{t\ge\delta:X(t)\in C\}$ of the process $\process{X}$ on  the set $C$  is finite and controlled by $\V(x)$ (see \cite[Theorem 4.1]{Douc-Fort-Guilin-2009}). 
However, this property alone  does not immediately imply ergodicity of $\process{X}$.
We must also  ensure that a similar property  holds for any  other ``reasonable'' set. 
If $\process{X}$ is $\uppsi$-irreducible and  $C$ is a petite
set, then   the
$\varphi\circ\Phi^{-1}$-moment of  $\tau_B^\delta$ is    finite and controlled by $\V(x)$ for any $B\in\mathfrak{B}([0,\infty))$ with $\uppsi(B)>0$    
(see \cite[the discussion after Theorem 4.1]{Douc-Fort-Guilin-2009}).
Nevertheless, as in the discrete setting, $\process{X}$ can exhibit certain cyclic behavior that  precludes ergodicity  (see \cite[Section 5]{Meyn-Tweedie-AdvAP-II-1993} and \cite[Chapter 5]{Meyn-Tweedie-Book-2009}). By assuming aperiodicity, which excludes this type of behavior, we can conclude that the subgeometric ergodicity of $\process{X}$ with rate $\varphi(\Phi^{-1}(t))$ of $\process{X}$ follows from    finiteness of the $\varphi\circ\Phi^{-1}$-moment of $\tau_C^\delta$ (see \cite[Theorem 1]{Fort-Roberts-2005}).  On the other hand, the   proof of  part (ii) of the theorem is again based on a variant of  the Foster-Lyapunov method for (sub)geometric lower bounds for ergodicity of Markov processes, as developed in \cite{Bresar-Mijatovic-2024}. The key step is to construct an appropriate Lyapunov function $\V:[0,\infty)\to[1,\infty)$ such that $x\mapsto1/\V(x)$ and $x\mapsto1/\bar\upnu(\V(x))$ are contained in the domain of  $\mathcal{L}^{e}$, and \begin{equation*}
\label{eq:lyap2}
\mathcal{L}^{e}(1/\V)(x)\le \frac{r(\V(x))}{x}+\chi\Ind_{\{V(x)\le\ell\}}\quad\text{and}\quad   \mathcal{L}^{e}(1/\bar\upnu(\V(x)))\ge-\chi\Ind_{\{V(x)\le\ell\}},\qquad x\ge0,\end{equation*} for some $\chi,\ell\in(0,\infty)$ (see \cite[Theorems 2.1 and 2.2]{Bresar-Mijatovic-2024}. With this at hand, one concludes lower bounds on the tails of the return times of $\process{X}$ to petite sets. This, in turn, yields lower bounds on the tails of $\uppi$ and on the rate of convergence of the marginals of $\process{X}$ to $\uppi$.

In the following theorem, we discuss uniform ergodicity of $\process{X}$.

\begin{theorem}
	\label{prop:upper_bounds_super_linear} Assume $\textbf{(C1)-(C3)}$ (or $\textbf{(\={C}1)-(\={C}2)}$ and  $\textbf{(C3)}$ in the case when $\bar\upnu(0)<\infty$), that $\process{X}$ is  $\uppsi$-irreducible and aperiodic, and that every compact set is petite for $\process{X}$. 
	Further, assume  $\int_1^\infty \frac{\D u}{r(u)}<\infty$ and $\limsup_{u\to\infty}\int_0^\infty \frac{\bar\upnu(v)}{r(u+v)}\D v<1$.
	Then the process $\process{X}$ is uniformly ergodic.
	The tail behavior of the invariant measure is as follows.
	\begin{enumerate}
		\item[(i)] (Sub-geometric upper bounds) If it holds that $\bar\upnu(u)>0$ for all $u>0$, the mapping $u\mapsto  r(u)/u^{1+\epsilon}\bar\upnu(u)$ is non-decreasing for all sufficiently large $u$,
		and 
		$$
		\limsup_{u\to\infty}\frac{u^{1+\epsilon}\bar\upnu(u)}{r(u)}\int_0^\infty\frac{\bar\upnu(v)}{(u+v)^{1+\epsilon}\bar\upnu(u+v)}\D v<1
		$$ for some $\epsilon>0$,
		then 
		$$\limsup_{u\to\infty}\frac{r(u)}{u^{1+\epsilon}\bar\upnu(u)}\bar\uppi(u)<\infty.$$
		\item[(ii)] (Exponential upper bounds) If $\lim_{u\to\infty}r(u)=\infty$  and $\limsup_{u\to\infty}\E^{cu}\bar\upnu(u)<\infty$, then for every $\epsilon\in(0,c)$ it holds that $$\limsup_{u\to\infty}r(u)\E^{(c-\epsilon)u}\bar\uppi(u)<\infty.$$
	\end{enumerate}
\end{theorem}
\noindent Note that the assumption in Theorem~\ref{prop:upper_bounds_super_linear}~(i) cannot be satisfied for any L\'evy measure $\upnu$ with  $\limsup_{u\to\infty}\E^{cu^a}\bar\upnu(u)<\infty$ for some $a>1$ and $c\in(0,\infty)$.  In this case, Theorem~\ref{prop:upper_bounds_super_linear}~(ii) yields exponential upper bounds on the tail of $\uppi$.
We note also that the lower-bound theory in~\cite{Bresar-Mijatovic-2024} is based on all petite sets of $\process{X}$ being bounded, making it 
not applicable in the case of uniform ergodicity (covered by Theorem~\ref{prop:upper_bounds_super_linear}) where the entire state space is petite.

In the case when $\bar\upnu(0)<\infty$, under $\textbf{(\={C}1)-(\={C}2)}$ and  $\textbf{(C3)}$, it has been shown in \cite[Theorem 4.2]{Meyn-Tweedie-AdvAP-II-1993}  together with \cite[Theorem 5.1]{Tweedie-1994} that $\process{X}$ is  $\uppsi$-irreducible, aperiodic and every compact set is petite for $\process{X}$. In the following proposition, we provide sufficient conditions for these properties to hold in the case when $\bar\upnu(0)=\infty$.
Define, $$R(u):=\sup_{v\ge0}(r(v)-r(v+u))\qquad\text{and}\qquad N(u):=u^2\inf_{0\le v\le u}\left(\upnu\wedge(\updelta_v\ast\upnu)\right)[0,\infty),\quad u>0.$$
Here, for measures $\upmu$ and $\upeta$ (on $\mathfrak{B}([0,\infty))$), $\upmu\wedge\upeta$ is defined as $\upmu-(\upmu-\eta)^+$, where $(\upmu-\eta)^+$ denotes the positive part in the Hahn-Jordan decomposition of the signed measure $\upmu-\upeta$, and $\upmu\ast\upeta$ denotes the  convolution measure of $\upmu$ and $\upeta$.
For $u_0>0$, define $$ \bar{\V}(u):=\int_{u_0}^u\frac{\D v}{r(v)}+1.$$   Clearly, $\bar{\V}(u)$ is continuous on $[u_0,\infty)$ and continuously differentiable on $(u_0,\infty)$. Let $\V:[0,\infty)\to[0,\infty)$ be continuously differentiable, and such that $\V(u)=\bar{\V}(u)$ for $u\ge u_0$. 

\begin{proposition}\label{p:IRR} In addition to $\textbf{(C1)-(C2)}$, assume that $\bar\upnu(0)=\infty$ and the following conditions hold:
	\begin{itemize}
		\item [(i)] $\displaystyle\int_0^1\frac{u}{N(u)}\D u<\infty\quad \text{and}\quad \displaystyle\limsup_{u\to0}\left(R(u)\int_u^1\frac{\D v}{N(v/2)}\right)<\frac{1}{2}$;
		\item[(ii)] 
		 $\displaystyle\int_{[1,\infty)}\V(u+v)\upnu(\D v)<\infty$ for all $u\ge0$ and there exists $u_0>0$ such that
		$\displaystyle\int_0^\infty\frac{\bar\upnu(v)}{r(u+v)}\D v\le1$ for all $u\ge u_0$.
	\end{itemize} Then, $\process{X}$ is  $\uppsi$-irreducible, aperiodic and every compact set is petite for $\process{X}$.
	\end{proposition}
\noindent For example, as discussed in \cite[Remark 3.2]{Liang-Schilling-Wang-2020}, the conditions in (i) 
will be satisfied if $$\upnu(\D u)\ge \frac{\theta}{u^{1+\alpha} }\Ind_{(0,1)}(u)\D u$$ for some  $\alpha\in(0,1)$ and $\theta>0$.

When $\process{X}$ is not regular enough (i.e., it is either not $\uppsi$-irreducible or aperiodic), the topology induced by the total variation distance becomes too ``rough'', meaning that it cannot fully capture the singular behaviour of $\process{X}$. In other words, $p(t,x,\D y)$ may not converge to the underlying invariant probability measure  (if it exists) in this topology, but rather in a weaker sense  (see \cite{Sandric-RIM-2017} and the references therein). 
Therefore, in such cases, we naturally resort to  Wasserstein distances, which, in a certain sense, induce a finer topology. Specifically, convergence with respect to a Wasserstein distance implies  weak convergence of probability measures (see \cite[Theorems 6.9 and 6.15]{Villani-Book-2009}).

\begin{theorem}\label{tm:WASS1}
Assume $\textbf{(C1)-(C3)}$ (or $\textbf{(\={C}1)-(\={C}2)}$ and  $\textbf{(C3)}$ in the case when $\bar\upnu(0)<\infty$).  Let  $p\ge1$ and $\beta:[0,\infty)\to[0,\infty)$ 
	be such that
	\begin{itemize}
		\item [(i)] $\displaystyle \int_{[0,\infty)}u\vee u^p\upnu(\D y)<\infty$;
		\item[(ii)] $\beta(t)$ is convex and $\beta(t)=0$ if, and only if, $t=0$;
		\item[(iii)] there exists $\Gamma>0$  such that
		\begin{equation}\label{eqWASS1} r(u)-r(v)\le
		-\Gamma\beta(v-u), \qquad 0\le u\le v.\end{equation}
	\end{itemize}  
	Then, the process $\process{X}$ admits a unique invariant  $\uppi\in\mathcal{P}_p$, and for any $\kappa>0$ and $\upmu\in\mathcal{P}_p$ it holds that: 
\begin{equation}\label{eqWASS2}\W_p\left(\int_{[0,\infty)}p(t,x,\D y)\upmu(\D x),\uppi(\D y)\right)\le \left(\frac{\W_p(\upmu,\uppi)}{\kappa}+1\right)B_\kappa^{-1}(\Gamma t),\qquad t\geq 0,\end{equation} 
 where $$B_\kappa(t):=\int_t^\kappa\frac{\D s}{\beta(s)},\qquad t\in(0,\kappa].$$ Furthermore, $$\limsup_{u\to\infty}u^p\bar\uppi(u)<\infty.$$
\end{theorem}

\subsection{Literature review}
Our work is situated within the ongoing research on the ergodicity properties of Markov processes and contributes to the extensive literature on stochastic differential equations. To highlight some of the most recent and pertinent studies, the ergodicity properties of diffusion processes with respect to the total variation distance have been established in \cite{Arapostathis-Borkar-Ghosh_Book-2012}, \cite{Kulik-Book-2018}, \cite{Lazic-Sandric-2021}, \cite{Lazic-Sandric-2022}. Ergodicity results for jump-diffusion processes and general Markov processes (also with respect to the total variation distance) can be found in \cite{Arapostathis-Pang-Sandric-2019}, \cite{Deng-Schilling-Song-2017, Deng-Schilling-Song-Erratum-2018} \cite{Masuda-2007, Masuda-Erratum-2009}, \cite{Meyn-Tweedie-AdvAP-II-1993}, \cite{Meyn-Tweedie-AdvAP-III-1993}, \cite{Sandric-ESAIM-2016}, \cite{Wang-2011}, and \cite{Wee-1999}.

As mentioned earlier, studies on ergodicity properties with respect to the total variation distance generally assume that the Markov process is $\uppsi$-irreducible, aperiodic, and that every compact set is petite. These assumptions hold when the process does not exhibit singular behavior in its motion. Together with the Lyapunov condition, which ensures the controllability of modulated moments of return times to a petite set, these conditions lead to the ergodic properties stated in the literature.
For Markov processes that do not converge in total variation, ergodic properties under Wasserstein distances are of interest, as these processes may converge weakly under certain conditions. For example, see \cite{Bolley-Gentil-Guillin-2012}, \cite{Butkovsky-2014}, \cite{Eberle-2011}, \cite{Eberle-2015}, \cite{Hairer-Mattingly-Scheutzow-2011}, \cite{Luo-Wang-2016}, \cite{Majka-2017}, \cite{Renesse-Sturm-2005}, and \cite{Wang-2016}.

In modern literature, the process studied in this paper falls within the class of shot-noise models. These models are widely studied due to their applications in queueing theory~\cite{MR4951813}, server systems, natural disaster modelling, and financial markets; see the survey in~\cite{MR4331112} for further details. Notably, recent work has established results on large deviations~\cite{MR4940338} and the distribution of the return times~\cite{MR3975897}. Our work complements these studies by analysing a related model and establishes new results on ergodicity.

In this paper, we explore transience, null recurrence, positive recurrence, as well as  ergodicity properties in terms of both the total variation distance and Wasserstein distance of a storage process defined by \eqref{eq1}, with a general release rule and cumulative inputs. This model was originally introduced in \cite{Moran-1956}. Its various properties, such as the existence and uniqueness of the solution to \eqref{eq1}, emptiness times, local times, wet period times, etc., have been studied in \cite{Brockwell-Chung-1975}, \cite{Cinlar-1975}, \cite{Cinlar-Pinsky-1971}, \cite{Cinlar-Pinsky-1972}, and \cite{Gani-Prabhu-1963}. Long-term behavior properties, including transience, recurrence, null recurrence, positive recurrence, and ergodic properties, have been investigated in \cite{Brockwell-1977}, \cite{Brockwell-Resnick-Tweedie-1982}, \cite{Harrison-Resnick-1976}, \cite{Meyn-Tweedie-AdvAP-III-1993}, \cite{Tuominen-Tweedie-1979}, \cite{Yamazato-2000}, and \cite{Yamazato-2005}. Specifically, \cite{Harrison-Resnick-1976} discusses positive recurrence (the existence of a stationary distribution) in the case where $\bar\upnu(0)<\infty$ and $\textbf{({C}2)}$ holds. This result was further extended in \cite{Brockwell-1977} to the case of general cumulative inputs. Studies on transience, null recurrence and positive recurrence are presented in \cite{Yamazato-2000} and \cite{Yamazato-2005}.
The first results on ergodicity and geometric ergodicity with respect to the total variation distance were obtained in \cite{Brockwell-Resnick-Tweedie-1982} and \cite{Tuominen-Tweedie-1979} using renewal theory, and in \cite[Section 9]{Meyn-Tweedie-AdvAP-III-1993} using the Lyapunov method. In this article, we build on these results by presenting conditions for transience, null recurrence, positive recurrence and both subgeometric and geometric ergodicity with respect to the total variation and Wasserstein distances. Our approach combines the Lyapunov method (for ergodicity with respect to the total variation distance) and the synchronous coupling method (for ergodicity with respect to the Wasserstein distance).

We also note that this work provides conditions for the $\uppsi$-irreducibility, aperiodicity, and petiteness of compact sets in storage processes, in the case when the input process does not have a finite jump rate (see Proposition \ref{p:IRR}). These properties, in the case of the input process with a finite jump rate, are discussed in \cite[Theorem 4.2]{Meyn-Tweedie-AdvAP-II-1993}.

\section{Proof of the main results}\label{s2}

In this section, we present the proofs of the main results of the article.

\subsection{Ergodicity with respect to the total variation distance}

We first establish Theorem~\ref{prop:transience}.

\begin{proof}[Proof of Theorem \ref{prop:transience}]  We first address part (i). Pick $\V:[0,\infty)\to[0,\infty)$ such that $\V(u) = u$ for $u\ge1$, and $1/\V:[0,\infty)\to[0,\infty)$ is continuously differentiable. 
	\begin{itemize}
		\item[(a)] We obtain $\mathcal{L}(1/\V)(u)\leq 0$ for all sufficiently large $u$.
		\item[(b)] The change of variables $v/u = z$ yields
		\begin{align*}
		\mathcal{L}(1/\V)(u) &= \frac{r(u)}{u^2} - \int_0^\infty\frac{\bar\upnu(v)}{(u+v)^2}\D v\\
		&= \frac{r(u)}{u^2} \left(1-\frac{u}{r(u)}\int_0^\infty \frac{\bar\upnu(uz)}{(1+z)^2}\D z\right),
		\end{align*} which is less than or equal to $0$ for all sufficiently large $u$.
		\end{itemize}
In both cases, the assertion follows from~\cite[Theorem~3.1]{MR4663508}. 

Next, we consider part (ii). Consider the same function as in the proof of part (i). We have $\mathcal{L}\V(u) = 0$ for all sufficiently large $u$. By c\`{a}dl\`{a}g version of~\cite[Lemma~3.2]{MR4852005} recurrence follows. To prove null recurrence, note that
$$
\mathcal{L}(1/\V)(u)\leq 1/\V(u)^2 = \varphi(1/\V(u))\quad\text{and}\quad \mathcal{L}V(u)\geq 0
$$ for all sufficiently large $u$.
Thus the \textbf{L}-drift condition in~\cite{Bresar-Mijatovic-2024} is satisfied with $(\V,\varphi,\Psi)$, where $\varphi(1/u) = 1/u^2$ and $\Psi(u) = u$. By~\cite[Theorem~2.5~(b)]{Bresar-Mijatovic-2024} the process is null recurrent, since return times to compact (petite) sets do not admit finite first moments.
	\end{proof}

We next prove Theorem \ref{tm:TV}.

\begin{proof}[Proof of Theorem \ref{tm:TV}] We begin with part (i).
	According to \eqref{eq2}, there exist $u_0>0$ and $\varepsilon\in(0,1)$, such that $$\mathcal{L}\V(u)\le -\varepsilon\varphi(\V(u)),\qquad u\ge u_0.$$ 
	Assume first that $\lim_{t\to\infty}\varphi'(t)=0$.
Thus, we have obtained the relation in (3.11) of \cite[Theorem 3.4 (i)]{Douc-Fort-Guilin-2009} with $\phi(t)=\varphi(t)$, $C=[0,u_0]$, and $b=\sup_{u\in C}|\mathcal{L}\V(u)|$. Now,  applying \cite[Theorem 3.2]{Douc-Fort-Guilin-2009} with $\Psi_1(t)=t$ and $\Psi_2(t)=1$ implies the desired result.  

Next, assume that $\liminf_{t\to\infty}\varphi'(t)>0$. In this case, as we have already commented, there exist $\Gamma\ge \gamma>0$ such that $$\gamma t-\gamma+\varphi(1)\le\varphi(t)\le\Gamma t-\Gamma+\varphi(1),\qquad t\ge1.$$   Thus,
$$\mathcal{L}\V(u)\le -\epsilon \V(u),\qquad u\ge u_1,$$ for some $\epsilon,u_1>0$,  which is exactly 	the Lyapunov equation on 
\cite[page 1679]{Down-Meyn-Tweedie-1995}
with $c=\epsilon$, $C=[0,u_1]$  and $b=\sup_{u\in C}|\mathcal{L}\V(u)|$.  However, observe that in \cite{Down-Meyn-Tweedie-1995}, a slightly different definition of aperiodicity is used. From 
\cite[Proposition 6.1]{Meyn-Tweedie-AdvAP-II-1993}, \cite[Theorem 4.2]{Meyn-Tweedie-AdvAP-III-1993}, and 
aperiodicity it follows that the is a petite set $C\subset[0,\infty)$, 	 $T>0$, and a non-trivial measure $\upeta_C$ on $\mathfrak{B}([0,\infty))$, such that $\upeta_C(C)>0$ and $$p(t,x,B)\ge\upeta_C(B),\qquad x\in C,\quad t\ge T,\quad B\in\mathfrak{B}([0,\infty)).$$ In particular, 
$$p(t,x,C)>0,\qquad x\in C,\quad t\ge T,$$	
which is exactly the definition of aperiodicity used in \cite[page 1675]{Down-Meyn-Tweedie-1995}.  The assertion now follows from \cite[Theorem 5.2]{Down-Meyn-Tweedie-1995}. 

Finally, let us show \eqref{tail}. First, according to \cite[Proposition 3.1 and Theorem 3.2]{Douc-Fort-Guilin-2009} and \cite[Theorem 5.2]{Down-Meyn-Tweedie-1995}, it holds that $$\int_{[0,\infty)}\varphi(\V(u))\uppi(\D u)<\infty\qquad\text{and}\qquad \lim_{t\to\infty}\mathbb{E}^x[\varphi(\V(X(t)))]=\int_{[0,\infty)}\varphi(\V(u))\uppi(\D u),\quad x\ge0.$$
Consequently, for fixed $x\ge0$ we have (recall that the function $u\mapsto\varphi(\V(u))$ is non-decreasing for $u\ge1$)
\begin{align*}&\limsup_{u\to\infty}(\varphi(\Phi^{-1}(u))\vee\varphi(\V(u)) )\bar\uppi(u)\\
&=\limsup_{u\to\infty}(\varphi(\Phi^{-1}(u))\vee\varphi(\V(u)) )\left(\bar\uppi(u)-p(u,x,(u,\infty))+p(u,x,(u,\infty))\right)\\
&\le \limsup_{u\to\infty}(\varphi(\Phi^{-1}(u))\vee\varphi(\V(u)) ) \left(\lVert p(u,x,\D y)-\uppi(\D y)\rVert_{{\rm TV}}+ \Prob^x(\varphi(\V(X(u)))\ge\varphi(\V(u)))\right)\\
&\le \limsup_{u\to\infty}(\varphi(\Phi^{-1}(u))\vee\varphi(\V(u)) ) \left(\lVert p(u,x,\D y)-\uppi(\D y)\rVert_{{\rm TV}}+ \frac{\mathbb{E}^x[\varphi(\V(X(u)))]}{\varphi(\V(u))}\right)
\\&\le \int_{[0,\infty)}\varphi(\V(u))\uppi(\D u).
\end{align*}

We proceed with part (ii). 	Pick $\V:[0,\infty)\to[0,\infty)$ such that $\V(u) = u$ for $u\ge1$, and $1/\V:[0,\infty)\to[0,\infty)$ is continuously differentiable.  Then we have 
$$
\mathcal{L}(1/\V)(u) = \frac{r(u)}{u^2} - \int_0^\infty \frac{\bar\upnu(v)}{(u+v)^2}\D v\leq \frac{r(u)}{u^2}
$$  for all sufficiently large $u$.
Define $\Psi(u) \coloneqq 1/\bar\upnu(u)$ and $\varphi(1/u) \coloneqq r(u)/u^2$. An argument similar to that in~\cite[Proof of Proposition~6.5]{Bresar-Mijatovic-2024} shows that $(\V,\varphi,\Psi)$ satisfies the $\mathbf{L}$-drift condition of~\cite{Bresar-Mijatovic-2024}. The lower bound on the tail of $\uppi$ then follows from~\cite[Theorem~2.1]{Bresar-Mijatovic-2024}, while the lower bound on the convergence rate to stationarity follows from\cite[Theorem~2.2]{Bresar-Mijatovic-2024}.

We conclude with part (iii). Choose $\V:[0,\infty)\to[0,\infty)$ such that $\V(u) = \log u$ for all $u$ large, and $1/\V:[0,\infty)\to[0,\infty)$ is continuously differentiable.  Then we have 
$$
\mathcal{L}(1/\V)(u) = \frac{r(u)}{u(\log u)^2} - \int_0^\infty \frac{\bar\upnu(v)}{(u+v)(\log (u+v))^2}\D v\leq \frac{r(u)}{u(\log u)^2}
$$for all sufficiently large $u$.
Define $\Psi(u) \coloneqq 1/\bar\upnu(\exp(u))$ and $\varphi(1/u) \coloneqq 1/u^2$. An argument similar to that in~\cite[Proof of Proposition~6.5]{Bresar-Mijatovic-2024} shows that $(\V,\varphi,\Psi)$ satisfies the $\mathbf{L}$-drift condition of~\cite{Bresar-Mijatovic-2024}. The lower bound on the tail of $\uppi$ then follows from~\cite[Theorem~2.1]{Bresar-Mijatovic-2024}, while the lower bound on the convergence rate to stationarity follows from\cite[Theorem~2.2]{Bresar-Mijatovic-2024}.
\end{proof}

Next, we turn to the proof of Theorem~\ref{prop:upper_bounds_super_linear}.

\begin{proof}[Proof of Theorem \ref{prop:upper_bounds_super_linear}] 
	Let $\V:[0,\infty)\to[1,\infty)$ be continuously differentiable   and such that $\V(u)=1+\int_1^u\D v/r(v)$ for $u\ge1$.
	The assumptions in the theorem imply  that
	$$\mathcal{L}\V(u)\leq -c\V(u)$$ for some $c>0$ and all $u>0$ sufficiently large. Uniform ergodicity now follows from~\cite[Theorem~5.2]{Down-Meyn-Tweedie-1995}.
	
	Next, we address the upper bounds on the tail of $\uppi$.
	\begin{itemize}
		\item[(i)] Let $\V:[0,\infty)\to[1,\infty)$ be continuously differentiable   and such that $\V(u) = 1+\int_1^u\D v/v^{1+\epsilon}\bar\upnu(v)$ for $u\ge1$. We obtain 
		$$
		\mathcal{L}\V(u)\leq -c \frac{r(u)}{u^{1+\epsilon}\bar\upnu(u)}$$ for some $c>0$ and all $u$ sufficiently large.
		This upper bound leads to a typically super-linear function $\phi$ in~\cite[Proposition~3.1]{Douc-Fort-Guilin-2009}. However, since the upper bound on the modulated moment in~\cite[Theorem~4.1~(i)]{Douc-Fort-Guilin-2009} requires a drift condition but not concavity of $\phi$, the conclusion of~\cite[Theorem~4.1~(i)]{Douc-Fort-Guilin-2009} holds under the drift condition in the previous display. 
		Let $f:[0,\infty)\to[0,\infty)$ be measurable and such that $\sup_{u\in[0,1]}f(u)<\infty$ and $f(u)=r(u)/u^{1+\varepsilon}\bar\upnu(u)$ for $u\ge1$. According to \cite[Theorem 4.1]{Meyn-Tweedie-AdvAP-III-1993} it then follows that
		$$\int_{[0,\infty)}f(u)\uppi(\D u)<\infty.$$ Hence, for sufficiently large $u$ and arbitrary $t\ge0$, $$\bar\uppi(u)\le\Prob^\uppi(f(X(t)))\ge f(u))\le\frac{\mathbb{E}^\uppi[f(X(t))]}{f(u)}=\frac{\int_{[0,\infty)}f(u)\uppi(\D u)}{f(u)}.$$
		\item[(ii)] 		Let   $\V(u) \df \E^{(\epsilon-c)u}$. We then have 
		\begin{align*}
		\mathcal{L}\V(u) &= -(c-\epsilon)\E^{(c-\epsilon)u}\left(r(u) - \int_0^\infty \E^{(c-\epsilon)v}\bar\upnu(v)\D v\right)\\
		&\leq -\frac{c-\epsilon}{2}r(u)\E^{(c-\epsilon)u}
		\end{align*}
		for all sufficiently large $u$.
		The assertion now follows by the same arguments as in part~(a).	\end{itemize}
\end{proof}

Finally, we prove
Proposition \ref{p:IRR}.

\begin{proof}[Proof of Proposition \ref{p:IRR}] 
First, according to \cite[Theorem 3.1]{Liang-Schilling-Wang-2020}, the conditions in (i)  of the statement of the proposition imply that for any $t>0$ and any bounded measurable function $f:[0,\infty)\to\R$, the function $$[0,\infty)\ni x\mapsto\int_{[0,\infty)}f(y)p(t,x,\D y) $$ is continuous and bounded. According to \cite[Proposition 6.1.1]{Meyn-Tweedie-Book-2009}, this implies that for any $t>0$ and $B\in\mathfrak{B}([0,\infty))$, $$\liminf_{y\to x} p(t,y,B)\ge p(t,x,B), \qquad x\ge0.$$ Fix an  arbitrary $t_0>0$ and define a kernel $T:[0.\infty)\times \mathfrak{B}([0,\infty))\to [0,1]$ by $$T(x,B):=p(t_0,x,B).$$ Thus, $\process{X}$ is a T-process in the sense of \cite[Section 1]{Tweedie-1994}. From \cite[Theorem 3.2]{Tweedie-1994} we next conclude that $\process{X}$ will be $\uppsi$-irreducible  if \begin{equation}\label{eq:tau}\Prob^x(\tau_{B(u_0,\varepsilon)}<\infty)>0\end{equation} for all $x\ge0$ and $\varepsilon>0$. Recall that $u_0>0$ is given in the statement of the proposition. Here, for $u\ge0$ and $\kappa>0$, $B(u,\kappa)=(u-\kappa,u+\kappa)\cap[0,\infty)$ and $\tau_{B(u,\kappa)}=\inf\{t\ge0:X(t)\in B(u,\kappa)\}.$ Clearly, \eqref{eq:tau} trivially holds for all $x\in B(u_0,\varepsilon).$ Now assume $0\le x\le u_0-\varepsilon$. We will show that $\Prob^x(X(t)\in B(u_0,\varepsilon))>0$ for all $t>0$, which clearly implies \eqref{eq:tau}.  We have \begin{align*}
\Prob^x(X(t)\in B(u_0,\varepsilon))&=\Prob(X(x,t)\in B(u_0,\varepsilon))\\
&=\Prob\left(x+A(t)-\int_0^tr(X(x,s))\D s\in B(u_0,\varepsilon)\right)\\
&=\Prob\left(A(t)\in B\left(u_0-x+\int_0^tr(X(x,s))\D s,\varepsilon\right)\right).
\end{align*}
Next, note that $$u_0-x+\int_0^tr(X(x,s))\D s\ge \varepsilon>0.$$ Hence, $$B\left(u_0-x+\int_0^tr(X(x,s))\D s,\varepsilon\right)\cap[0,\infty)\neq\emptyset.$$ The assertion now follows from \cite[Theorem 24.10]{Sato-Book-1999}, which states that $\Prob(A(t)\in O\cap[0,\infty))>0$ for every $t>0$ and every open set $O\subseteq\R$ with $O\cap[0,\infty)\neq\emptyset.$
We next show that for any $\varepsilon>0$ and  $x\in[u_0+\varepsilon,\infty)$, we have $$\Prob^x(\tau_{B(u_0,\varepsilon)}<\infty)>0.$$ Since $\upnu$ admits only positive jumps, we take a different approach from the previous argument. First, observe that from assumption (iii), we conclude that $$\mathcal{L}\V(u)=-1+\int_0^\infty\frac{\bar\upnu(v)}{r(u+v)}\D v\le0,\qquad u\ge u_0.$$
Next, for $\kappa>0$, let $f_\kappa\in\mathcal{C}_c^1([0,\infty])$ be such that $\Ind_{[0,\kappa]}(u)\le f_\kappa(u)\le\Ind_{[0,2\kappa]}(u)$.
According to \cite[Theorem 2.2.13 and Proposition 4.1.7]{Ethier-Kurtz-Book-1986}, we have
$$\mathbb{E}^x\left[(\V f_\kappa)(X(t\wedge\tau_{B(u_0,\varepsilon)}\wedge\tau_{[n,\infty)}))\right]-(\V f_\kappa)(x)=\mathbb{E}^x\left[\int_0^{t\wedge\tau_{B(u_0,\varepsilon)}\wedge\tau_{[n,\infty)}} \mathcal{L}(\V f_\kappa)(X(s))\D s\right]$$ for all $x\ge0$ and $t\ge0.$ In particular, $$\mathbb{E}^x\left[(\V f_\kappa)(X(t\wedge\tau_{[n,\infty)}))\Ind_{\{\tau_{B(u_0,\varepsilon)}>\tau_{[n,\infty)}\}}\right]-(\V f_\kappa)(x)\le\mathbb{E}^x\left[\int_0^{t\wedge\tau_{B(u_0,\varepsilon)}\wedge\tau_{[n,\infty)}} \mathcal{L}(\V f_\kappa)(X(s))\D s\right]$$for all $x\ge0$ and $t\ge0.$ By letting $\kappa\to\infty$, dominated and monotone convergence theorems yield 
$$\mathbb{E}^x\left[\V(X(t\wedge\tau_{[n,\infty)}))\Ind_{\{\tau_{B(u_0,\varepsilon)}>\tau_{[n,\infty)}\}}\right]-\V(x)\le\mathbb{E}^x\left[\int_0^{t\wedge\tau_{B(u_0,\varepsilon)}\wedge\tau_{[n,\infty)}} \mathcal{L}\V(X(s))\D s\right]$$ for all $x\ge0$ and $t\ge0.$ Consequently, for $x\ge u_0$,
$$\mathbb{E}^x\left[\V(X(t\wedge\tau_{[n,\infty)}))\Ind_{\{\tau_{B(u_0,\varepsilon)}>\tau_{[n,\infty)}\}}\right]\le\V(x),\qquad t\ge0.$$ By letting $t\to\infty$, Fatou's lemma implies that
$$\V(n)\Prob^x(\tau_{B(u_0,\varepsilon)}>\tau_{[n,\infty)})\le\V(x),\qquad x\ge u_0,$$ i.e.,
$$\Prob^x(\tau_{B(u_0,\varepsilon)}>\tau_{[n,\infty)})\le\frac{\V(x)}{\V(n)},\qquad x\ge u_0.$$ By letting $n\to\infty$, nonexplosivity implies  $$\Prob^x(\tau_{B(u_0,\varepsilon)}=\infty)\le\lim_{n\to\infty}\frac{\V(x)}{\V(n)}<1,\qquad x\ge u_0.$$
Finally, form \cite[Theorems 3.2 and 5.1]{Tweedie-1994},
we conclude that $\process{X}$ is $\uppsi$-irreducible, and every compact set is petite for $\process{X}$.

At the end, let us show that $\process{X}$ is aperiodic. We show that the skeleton chain $\{X(n)\}_{n\ge0}$ is $\upphi$-irreducible. To do so, we  apply \cite[Theorem 3.2]{Tweedie-1994}. The fact that $\{X(n)\}_{n\ge0}$ is a T-chain follows form the first part of the proof. Let $u_1>u_0$ and $\varepsilon>0$ be such that $2\varepsilon<u_1-u_0<u_0-2\varepsilon$. 
According to the first part of the proof, \begin{equation}\label{eq:irrr}\Prob^x(X(t)\in B(u_1,\varepsilon))>0,\qquad  x\in[0,u_1+\varepsilon),\quad t>0,\end{equation} and $$\Prob^x(\tau_{B(u_0,\varepsilon)}<\infty)>0,\qquad x> u_0-\varepsilon.$$
Now, observe that the second relation implies that for every $x> u_0-\varepsilon$, there exists $t=t(x)>0$ such that $p(t,x,B(u_0,\varepsilon))>0$. To see why, suppose for the sake of contradiction that there exists 
 $x> u_0-\varepsilon$  such that $p(t,x,B(u_0,\varepsilon))=0$ for all $t>0$. In this case, we would have: $$\mathbb{E}^x\left[\int_0^\infty \Ind_{\{X(t)\in B(u_0,\varepsilon)\}}\D t\right]=\int_0^\infty p(t,x,B(u_0,\varepsilon))\D t=0,$$ which implies $$\int_0^\infty \Ind_{\{X(t)\in B(u_0,\varepsilon)\}}\D t=0\qquad \Prob^x\text{-a.s.}$$
On the other hand, for every $\omega\in\{\omega:\tau_{B(u_0,\varepsilon)}(\omega)<\infty\}$, there exists $t=t(\omega)>0$ such that $X(t)(\omega)\in B(u_0,\varepsilon)$. Since $\process{X}$ is right-continuous and $B(u_0,\varepsilon)$ is open, we must have $$\int_0^\infty \Ind_{\{X(t)(\omega)\in B(u_0,\varepsilon)\}}\D t>0,$$ which contradicts the previous relation. Therefore, for every $x> u_0-\varepsilon$, there exists $t>0$ such that $p(t,x,B(u_0,\varepsilon))>0$. Now, let $n>t$, $n\in\N$, be arbitrary. We have the following: \begin{align*}p(n,x,B(u_1,\varepsilon))&=\int_{[0,\infty)} p(t,x,\D y) p(n-t,y,B(u_1,\varepsilon))\\
&\ge \int_{B(u_0,\varepsilon)}p(t,x,\D y) p(n-t,y,B(u_1,\varepsilon))\\
&>0.\end{align*}
Thus, we have shown that for every $x\ge u_1+\varepsilon$, there exists $n\in\N$ such that $$\Prob^x(X(n)\in B(u_1,\varepsilon))>0.$$
This, together with \eqref{eq:irrr}, implies that for every  $\varepsilon>0$ and $x\ge0$, we have $$\Prob^x(\tau^{c}_{B(u_1,\varepsilon)}<\infty)>0,$$ where $\tau^{c}_{B(u_1,\varepsilon)}:=\inf\{n\ge0:X(n)\in B(u_1,\varepsilon)\}.$ The assertion now follows from \cite[Theorem 3.2]{Tweedie-1994}.
\end{proof}

We now apply  Theorems \ref{prop:transience}, \ref{tm:TV} and \ref{prop:upper_bounds_super_linear} to standard storage models presented in Section \ref{s1}. We first consider the storage model with a constant release rate.

\renewcommand{\customname}{Constant release rate}
\begin{custom}
Assume that $r(u)=a\Ind_{(0,\infty)}(u)$ for $a>0$.  
According to Theorem \ref{prop:transience}, $\process{X}$ is transient if $m_\nu>a$ or $$\liminf_{u\to\infty}u\int_0^\infty\frac{\bar\upnu(uv)}{1+v^2}\D v>a,$$ while it is null recurrent if $m_\nu=a$.
If $\varphi(t)=ct$ for some $c>0$, then clearly $\Phi(t)=(\ln t)/c$, $\bar{\V}(u)=\E^{c((u-1)/a)+1}$, and \eqref{tm:TV} becomes $$\int_0^\infty\E^{\frac{cv}{a}}\bar\upnu(v)\D v<a.$$ Assume that $\int_{[1,\infty)}\E^{cu/a}\upnu(\D u)<\infty$. Then, 
$\process{X}$ will be geometrically  ergodic with rate $\E^{ct}$ if $$
\int_{[0,\infty)}\left(\E^{\frac{cu}{a}}-1\right)\upnu(\D u)<c.$$

Assume next that $\varphi(t)=t^\alpha$ for some $\alpha\in(0,1)$. Then,  $\Phi(t)=(t^{1-\alpha}-1)/(1-\alpha)$, $$\bar{\V}(u)=\left(\frac{1-\alpha}{a}u+\frac{\alpha-1+2a-a\alpha}{a}\right)^{1/(1-\alpha)},$$ and \eqref{tm:TV} will follow if \begin{align*}&\limsup_{u\to\infty}\frac{\int_0^\infty\left(u+v+\frac{\alpha-1+2a-a\alpha}{1-\alpha}\right)^{\alpha/(1-\alpha)}\bar\upnu(v)\D v}{\left(u+\frac{\alpha-1+2a-a\alpha}{1-\alpha}\right)^{\alpha/(1-\alpha)}}\\
&\le \limsup_{u\to\infty}\int_0^\infty\left(1+\frac{v}{u+\frac{\alpha-1+2a-a\alpha}{1-\alpha}}\right)^{\alpha/(1-\alpha)}\bar\upnu(v)\D v\\&<a.\end{align*} Assume that $\int_{[1,\infty)}u^{1/(1-\alpha)}\upnu(\D u)<\infty$. Then, 
$\process{X}$ will be subgeometrically  ergodic with rate $t^{\alpha/(1-\alpha)}$ if 
$
m_\upnu<a$. \qed
\end{custom}

We now discuss the storage model with a linear release rate.

\renewcommand{\customname}{Linear release rate}
\begin{custom}Assume that $r(u)=a+bu$ with $a\ge0$ and $b>0$. 	According to Theorem \ref{prop:transience}, $\process{X}$ is transient if  $$\liminf_{u\to\infty}\int_0^\infty\frac{\bar\upnu(uv)}{1+v^2}\D v>b.$$
	 If $\varphi(t)=ct$ for some $c>0$, then  $\Phi(t)=(\ln t)/c$, $\bar{\V}(u)=\E^c\left(\frac{bu+a}{b+a}\right)^{c/b}$, and \eqref{tm:TV} reduces to $$\limsup_{u\to\infty}\frac{\int_0^\infty (u+v+a/b)^{c/b-1}\bar\upnu(v)\D v}{(u+a/b)^{c/b}}<b.$$ Assume that $\int_{[1,\infty)} u^{c/b}\upnu(\D u)<\infty$. Then, if $c\le b$, we have
	$$\limsup_{u\to\infty}\frac{\int_0^\infty (u+v+a/b)^{c/b-1}\bar\upnu(v)\D v}{(u+a/b)^{c/b}}\le\limsup_{u\to\infty}\frac{\int_0^\infty v^{c/b-1}\bar\upnu(v)\D v}{(u+a/b)^{c/b}}=0.$$ On the other hand, if $c>b$, then 
	$$\limsup_{u\to\infty}\frac{\int_0^\infty (u+v+a/b)^{c/b-1}\bar\upnu(v)\D v}{(u+a/b)^{c/b}}\le\limsup_{u\to\infty}\frac{(1\vee 2^{c/b-2})\int_0^\infty v^{c/b-1}\bar\upnu(v)\D v}{(u+a/b)^{c/b}}=0.$$ Hence, in both cases $\process{X}$ is geometrically  ergodic with rate $\E^{ct}$.

Now,  let  
$\varphi(t)=t^\alpha$ for some $\alpha\in(0,1)$. Then,  $\Phi(t)=(t^{1-\alpha}-1)/(1-\alpha)$ and  $$\bar{\V}(u)=\left(\frac{1-\alpha}{b}\ln\left(\frac{bu+a}{b+a}\right)+2-\alpha\right)^{1/(1-\alpha)}.$$
If $\int_{[1,\infty)}(\ln(u))^{1/(1-\alpha)}\upnu(\D u)<\infty$, 
then
\begin{align*}&\limsup_{u\to\infty}\frac{\int_0^\infty\frac{\left(\frac{1-\alpha}{b}\ln \left(\frac{bu+bv+a}{b+a}\right)+2-\alpha\right)^{\alpha/(1-\alpha)}}{bu+bv+a}\bar\upnu(v)\D v}{\left(\frac{1-\alpha}{b}\ln \left(\frac{bu+a}{b+a}\right)+2-\alpha\right)^{\alpha/(1-\alpha)}}\\
&\le \frac{4^{\alpha/(1-\alpha)}}{b}\limsup_{u\to\infty}\frac{\int_0^\infty\frac{\left(\ln \left(u+v+\frac{a}{b}\right)\right)^{\alpha/(1-\alpha)}}{u+v+\frac{a}{b}}\bar\upnu(v)\D v}{\left(\ln \left(u+\frac{a}{b}\right)\right)^{\alpha/(1-\alpha)}}\\
&= \frac{4^{\alpha/(1-\alpha)}(1-\alpha)}{b}\limsup_{u\to\infty}\frac{\int_0^\infty\left(\left(\ln\left(u+v+\frac{a}{b}\right)\right)^{1/(1-\alpha)}-\left(\ln \left(u+\frac{a}{b}\right)\right)^{1/(1-\alpha)}\right)\upnu(\D v )}{\left(\ln \left(u+\frac{a}{b}\right)\right)^{\alpha/(1-\alpha)}}\\
&\le \frac{4^{\alpha/(1-\alpha)}}{b}\limsup_{u\to\infty}\frac{\int_0^\infty\ln\left(1+\frac{v}{u+a/b}\right)\left(\ln \left(u+\frac{a}{b}\right)+\ln\left(1+\frac{v}{u+a/b}\right)\right)^{\alpha/(1-\alpha)}\upnu(\D v )}{\left(\ln \left(u+\frac{a}{b}\right)\right)^{\alpha/(1-\alpha)}}\\
&=0.\end{align*}
Hence, \eqref{tm:TV} holds, and $\process{X}$ is subgeometrically  ergodic with rate $t^{\alpha/(1-\alpha)}$. \qed
\end{custom}

We now turn to the storage model with a power-type release rate.

\renewcommand{\customname}{Power-type release rate}
\begin{custom}  Assume that $r(u)=u^\beta$ for $\beta\neq0,1$. 	According to Theorem \ref{prop:transience}, $\process{X}$ is transient if  $$\liminf_{u\to\infty}u^{1-\beta}\int_0^\infty\frac{\bar\upnu(uv)}{1+v^2}\D v>1.$$
	 If $\varphi(t)=ct$ for some $c>0$, then  $\Phi(t)=(\ln t)/c$, $$\bar{\V}(u)=\E^{\frac{c}{1-\beta}(u^{1-\beta}-1)+c},$$ and \eqref{tm:TV} becomes $$\limsup_{u\to\infty}\frac{\int_0^\infty \frac{\E^{\frac{c}{1-\beta}(u+v)^{1-\beta}}}{(u+v)^\beta}\bar\upnu(v)\D v}{\E^{\frac{c}{1-\beta}u^{1-\beta}}}<1.$$
	If $\beta>1$, since 
	$$\limsup_{u\to\infty}\frac{\int_0^\infty \frac{\E^{\frac{c}{1-\beta}(u+v)^{1-\beta}}}{(u+v)^\beta}\bar\upnu(v)\D v}{\E^{\frac{c}{1-\beta}u^{1-\beta}}}\le \frac{2}{\beta-1}\limsup_{u\to\infty}\int_{[0,\infty)}\left(u^{1-\beta}-(u+v)^{1-\beta}\right)\upnu(\D v)=0,$$
	it follows that $\process{X}$ is geometrically  ergodic with rate $\E^{ct}$. Furthermore, according to Theorem \ref{prop:upper_bounds_super_linear}, if $$\lim_{u\to\infty}\int_0^\infty\frac{\bar\upnu(v)}{(u+v)^\beta}\D v=0,$$ $\process{X}$ is uniformly ergodic. If $0<\beta<1$ and $\int_{[1,\infty)}\E^{\frac{c}{1-\beta}u^{1-\beta}}\upnu(\D u)<\infty$, then 
	$$\limsup_{u\to\infty}\frac{\int_0^\infty \frac{\E^{\frac{c}{1-\beta}(u+v)^{1-\beta}}}{(u+v)^\beta}\bar\upnu(v)\D v}{\E^{\frac{c}{1-\beta}u^{1-\beta}}}\le \limsup_{u\to\infty}\int_0^\infty \frac{\E^{\frac{c}{1-\beta}v^{1-\beta}}}{(u+v)^\beta}\bar\upnu(v)\D v=0.$$ Hence, $\process{X}$ is again geometrically  ergodic with rate $\E^{ct}$.  On the other hand,  if $\beta<0$, we have
	\begin{align*}&\limsup_{u\to\infty}\frac{\int_0^\infty \frac{\E^{\frac{c}{1-\beta}(u+v)^{1-\beta}}}{(u+v)^\beta}\bar\upnu(v)\D v}{\E^{\frac{c}{1-\beta}u^{1-\beta}}}\\
	&\ge \limsup_{u\to\infty}\int_0^\infty \frac{\E^{\frac{c}{1-\beta}\left((u+v)^{1-\beta}-u^{1-\beta}\right)}}{(u+v)^\beta}\bar\upnu(v)\D v\\
	&\ge \limsup_{u\to\infty}\int_0^\infty \frac{\E^{cvu^{-\beta}}}{(u+v)^\beta}\bar\upnu(v)\D v\\
	&=\infty.\end{align*} Hence, we cannot conclude the geometric ergodicity of $\process{X}$ (if it holds at all) from Theorem \ref{tm:TV}.

	 Let now 
	$\varphi(t)=t^\alpha$ for some $\alpha\in(0,1)$. Then,  $\Phi(t)=(t^{1-\alpha}-1)/(1-\alpha)$ and  $$\bar{\V}(u)=\left(\frac{1-\alpha}{1-\beta}u^{1-\beta}-\frac{1-\alpha}{1-\beta}+2-\alpha\right)^{1/(1-\alpha)}.$$
For $\beta>1$,	we have \begin{align*} &\limsup_{u\to\infty} \frac{\int_0^\infty\frac{\left(\frac{1-\alpha}{1-\beta}(u+v)^{1-\beta}-\frac{1-\alpha}{1-\beta}+2-\alpha\right)^{\alpha/(1-\alpha)}}{(u+v)^\beta}\bar\upnu(v)\D v}{\left(\frac{1-\alpha}{1-\beta}u^{1-\beta}-\frac{1-\alpha}{1-\beta}+2-\alpha\right)^{\alpha/(1-\alpha)}}\\&\le \frac{\left(\frac{1-\alpha}{\beta-1}+2-\alpha\right)^{\alpha/(1-\alpha)}}{(\beta-1)(2-\alpha)^{\alpha/(1-\alpha)}}\limsup_{u\to\infty}\int_{[0,\infty)}\left(u^{1-\beta}-(u+v)^{1-\beta}\right)\upnu(\D v)=0.\end{align*} Hence, $\process{X}$ is subgeometrically  ergodic with rate $t^{\alpha/(1-\alpha)}$. Assume next that 
	 $\alpha\le\beta<1$ and $\int_{[1,\infty)}u\upnu(\D u)<\infty$, then
	\begin{align*}&\limsup_{u\to\infty} \frac{\int_0^\infty\frac{\left(\frac{1-\alpha}{1-\beta}(u+v)^{1-\beta}-\frac{1-\alpha}{1-\beta}+2-\alpha\right)^{\alpha/(1-\alpha)}}{(u+v)^\beta}\bar\upnu(v)\D v}{\left(\frac{1-\alpha}{1-\beta}u^{1-\beta}-\frac{1-\alpha}{1-\beta}+2-\alpha\right)^{\alpha/(1-\alpha)}}\\
	&\le 4\limsup_{u\to\infty} \frac{\int_0^\infty(u+v)^{\frac{\alpha(1-\beta)}{1-\alpha}-\beta}\bar\upnu(v)\D v}{u^{\frac{\alpha(1-\beta)}{1-\alpha}}}\\
	&\le \frac{4(1-\beta)}{1-\alpha}\limsup_{u\to\infty} \frac{\int_0^\infty \left((u+v)^{\frac{1-\beta}{1-\alpha}}-u^{\frac{1-\beta}{1-\alpha}}\right)\upnu(\D v)}{u^{\frac{\alpha(1-\beta)}{1-\alpha}}}\\
	&\le   \frac{4(1-\beta)^2}{(1-\alpha)^2}\limsup_{u\to\infty}  u^{\frac{\alpha\beta-\beta}{1-\alpha}}\int_0^\infty v\upnu(\D v)\\&=0. \end{align*} Hence,  in this case $\process{X}$ is again subgeometrically ergodic
	with rate $t^{\alpha/(1-\alpha)}$.
Finally, if $\beta<\alpha$ and $\int_{[1,\infty)}u^{(1-\beta)/(1-\alpha)}\upnu(\D u)<\infty$, then
	\begin{align*}&\limsup_{u\to\infty} \frac{\int_0^\infty\frac{\left(\frac{1-\alpha}{1-\beta}(u+v)^{1-\beta}-\frac{1-\alpha}{1-\beta}+2-\alpha\right)^{\alpha/(1-\alpha)}}{(u+v)^\beta}\bar\upnu(v)\D v}{\left(\frac{1-\alpha}{1-\beta}u^{1-\beta}-\frac{1-\alpha}{1-\beta}+2-\alpha\right)^{\alpha/(1-\alpha)}}\\
	&\le\frac{4(1-\beta)}{1-\alpha}\limsup_{u\to\infty} \frac{\int_0^\infty \left((u+v)^{\frac{1-\beta}{1-\alpha}}-u^{\frac{1-\beta}{1-\alpha}}\right)\upnu(\D v)}{u^{\frac{\alpha(1-\beta)}{1-\alpha}}}\\
	&\le \frac{4(1-\beta)^2}{(1-\alpha)^2}\limsup_{u\to\infty}  \frac{\int_0^\infty v(u+v)^{\frac{\alpha-\beta}{1-\alpha}}\upnu(\D v)}{u^{\frac{\alpha(1-\beta)}{1-\alpha}}}\\
	&=0,\end{align*}
	which implies subgeometrically ergodicity of $\process{X}$ 
	with rate $t^{\alpha/(1-\alpha)}$.    \qed
\end{custom}

We now consider the storage model with a ``general'' release rate. Related results can be found in \cite[Section 9]{Meyn-Tweedie-AdvAP-III-1993} and \cite[Theorem 4]{Tuominen-Tweedie-1979}. 
\renewcommand{\customname}{``General'' release rate}
\begin{custom}  Assume that $\liminf_{u\to\infty}r(u)>a$ for $a>0$.  Hence, there exists $u_0\ge0$ such that $r(u)\ge a$ for $u\ge u_0$.
If $\varphi(t)=ct$ for some $c>0$, then  $\Phi(t)=(\ln t)/c$, and  
	$$\limsup_{u\to\infty}\frac{\int_0^\infty\frac{\varphi(\V(u+v))}{r(u+v)}\bar\upnu(v)\D v}{\varphi(\V(u))}\le \limsup_{u\to\infty}\frac{\int_0^\infty\E^{c\int_u^{u+v}\frac{\D w}{r(w)}}\bar\upnu(v)\D v}{a}\le \frac{\int_0^\infty\E^{\frac{c}{a}v}\bar\upnu(v)\D v}{a}.$$ Hence, if $\int_{[1,\infty)}\E^{\frac{c}{a}u}\upnu(\D u)<\infty,$ and $$\int_0^\infty\E^{\frac{c}{a}u}\bar\upnu(v)\D u=\frac{a}{c}\int_{[0,\infty)}\left(\E^{\frac{c}{a}u}-1\right)\upnu(\D u)<a,$$ then  
	$\process{X}$ is geometrically ergodic with rate $\E^{ct}$.

Let now $\varphi(t)=t^\alpha$ for some $\alpha\in(0,1)$. Then,  $\Phi(t)=(t^{1-\alpha}-1)/(1-\alpha)$. If $\int_{[1,\infty)}u^{1/(1-\alpha)}\upnu(\D u)<\infty$, we have
\begin{align*}&\limsup_{u\to\infty}\frac{\int_0^\infty\frac{\varphi(\V(u+v))}{r(u+v)}\bar\upnu(v)\D v}{\varphi(\V(u))}\\
&\le\frac{1}{a} \limsup_{u\to\infty}\frac{\int_0^\infty\left(\int_1^{u+v}\frac{\D w}{r(w)}+\frac{2-\alpha}{1-\alpha}\right)^{\alpha/(1-\alpha)}\bar\upnu(v)\D v}{\left(\int_1^{u}\frac{\D w}{r(w)}+\frac{2-\alpha}{1-\alpha}\right)^{\alpha/(1-\alpha)}}\\
&\le \frac{1}{a} \limsup_{u\to\infty}\int_0^\infty\left(1+ \frac{\int_u^{u+v}\frac{\D w}{r(w)}}{\int_1^{u}\frac{\D w}{r(w)}+\frac{2-\alpha}{1-\alpha}}\right)^{\alpha/(1-\alpha)}\bar\upnu(v)\D v\\
&\le\frac{1}{a} \int_0^\infty\left(1+ \frac{v}{a\liminf_{u\to\infty}\int_1^{u}\frac{\D w}{r(w)}+\frac{(2-\alpha)a}{1-\alpha}}\right)^{\alpha/(1-\alpha)}\bar\upnu(v)\D v.\end{align*}
Hence, if $\liminf_{u\to\infty}\int_1^{u}\frac{\D w}{r(w)}=\infty,$
then 
	$\process{X}$ will be subgeometrically  ergodic with rate $t^{\alpha/(1-\alpha)}$ if
	$m_\upnu<a.$ On the other hand, 
	if $\liminf_{u\to\infty}\int_1^{u}\frac{\D w}{r(w)}=b$, then $\process{X}$ will be subgeometrically  ergodic with rate $t^{\alpha/(1-\alpha)}$ if
	\begin{align*}\pushQED{\qed} &\frac{1}{a} \int_0^\infty\left(1+ \frac{v}{ab+\frac{(2-\alpha)a}{1-\alpha}}\right)^{\alpha/(1-\alpha)}\bar\upnu(v)\D v\\
	&=(b(1-\alpha)+2-\alpha)\int_0^\infty\left(\left(1+\frac{u}{ab+\frac{(2-\alpha)a}{1-\alpha}}\right)^{1/(1-\alpha)}-1\right)\upnu(\D u)\\
	&<1.\qedhere
	\popQED\end{align*}
	\end{custom}

Finally, we consider several storage models for which we obtain a sharp quantitative characterization of the ergodic convergence rates by matching upper and lower bounds. Recall that for $\alpha \in \R\setminus\{0\}$ and a function $f:(0,\infty)\to(0,\infty)$, we write $f(u)\approx u^\alpha$ if  for every $\epsilon \in (0,|\alpha|)$, 
$$u^{\alpha-\epsilon}\preceq f(u)\preceq u^{\alpha+\epsilon}.$$

\renewcommand{\customname}{Sharp ergodic convergence rates}
\begin{custom} 
	\begin{itemize}
		\item[(i)] Assume $r(u) \approx u^\beta$ and $\bar\upnu(u) \approx u^{-\alpha}$ for some $\alpha,\beta \in (0,\infty)$. According to Theorem~\ref{prop:transience}, the process $\process{X}$ is transient if $\alpha + \beta < 1$.
		If $\alpha + \beta > 1$ and $\beta < 1$, then by taking $\varphi(t) \approx t^{\frac{\alpha+\beta-1}{\alpha}}$ and $h(u) \approx u^\alpha$, it follows from Theorem~\ref{tm:TV} that  
		$$
		\bar \uppi(u)\approx u^{1-\alpha-\beta} \qquad\text{and}\qquad \|p(t,x,\D y)-\uppi(\D y)\|_{\mathrm{TV}} \approx t^{\frac{1-\alpha-\beta}{1-\beta}}.
		$$
		Finally, if $\beta > 1$, Theorem~\ref{prop:upper_bounds_super_linear} yields uniform ergodicity of $\process{X}$, and the invariant measure satisfies $$\limsup_{u\to\infty}u^{\alpha+\beta-1-\epsilon}\bar\uppi(u) <\infty$$ for all $\epsilon>0$.
		
		\item[(ii)] Assume $\bar\upnu(u) \approx u^{-\alpha}$ for some $\alpha \in (1,\infty)$ and that $m_\upnu<\lim_{u\to\infty} r(u) <\infty$. By taking $\varphi(t) \approx t^{\frac{\alpha-1}{\alpha}}$ and $h(u) \approx u^\alpha$, it follows from Theorem~\ref{tm:TV} that
		\begin{equation*}\pushQED{\qed}
		\bar \uppi(u)\approx u^{1-\alpha} \qquad\text{and}\qquad \|p(t,x,\D y)-\uppi(\D y)\|_{\mathrm{TV}} \approx t^{1-\alpha}. \qedhere  \popQED
		\end{equation*}
	\end{itemize}
\end{custom}

\subsection{Ergodicity with respect to Wasserstein distances}\label{s3}

In this subsection, we prove Theorem \ref{tm:WASS1}. First, we establish an auxiliary result that will be used in its proof.

 \begin{lemma}\label{lm:WASS1} Let $\Delta,T>0$, and  let $f:[0,T)\to[0,\infty)$ satisfy the following  
 	conditions:
 	\begin{itemize}
 		\item [(i)] $f(t)$ is  absolutely continuous on any interval $[t_0,t_1]$ with $0<t_0<t_1<T$;
 		\item[(ii)] $f'(t)\leq -\Delta\beta(f(t))$ a.e. on $[0,T)$;
 		\item[(iii)] $\beta(f(t))>0$ a.e. on $[0,T)$.
 	\end{itemize}  
 	Then, for any $\kappa\ge f(0)$, we have $$f(t)\leq B_{\kappa}^{-1}(\Delta t),\qquad 0\le t< T.$$
 \end{lemma}
 \begin{proof}
 	By assumption, for any $t\in[0,T)$, $$-B_{f(0)}(f(t))=\int_{f(0)}^{f(t)}\frac{\D s}{\beta(s)}=\int_0^t\frac{f'(s)\,\D s}{\beta(f(s))}\le-\Delta t.$$ 
 	Hence, $$f(t)\leq B_{f(0)}^{-1}(\Delta t),\qquad 0\le t<\Delta^{-1}B_{f(0)}(0)\wedge T.$$
 	Now, the  assertion follows follows from the fact that 
 	$B_{f(0)}(t)\le B_\kappa(t)$ for all $t\in(0,f(0)]$, and from the relation $$\beta(t)=\beta(t+(1-t)0)\le t\beta(1)+(1-t)\beta(0)=t\beta(1),\qquad t\in[0,1],$$ which implies $B_{f(0)}(0)=\infty$.
 \end{proof}

We now prove Theorem \ref{tm:WASS1}.

\begin{proof}[Proof of Theorem \ref{tm:WASS1}] 
	Let us first remark that, according to \cite[Lemma 2.3]{Sandric-Arapostathis-Pang-2022}, we have $p(t,x,\D y)\in\mathcal{P}_p$ for all $x\ge0$ and $t>0$. Next, fix $x,y\ge0$, with $x\neq y$, and let $\process{X}$ and $\process{Y}$ be solutions to \eqref{eq1} starting from $x$ and $y$, respectively. Define $\tau:=\inf\{t>0:X(t)=Y(t)\}$ and 
	\begin{equation*}Z(t):=\left\{\begin{array}{cc}
	Y(t), & t<\tau, \\
	X(t),& t\ge \tau,
	\end{array}\right.\qquad t\ge0. \end{equation*}	By employing the strong Markov property, it is easy to see that $\mathbb{P}^y(Z(t)\in\cdot)=\mathbb{P}^y(Y(t)\in\cdot)$ for all $t\ge0$. Consequently,
	\begin{equation*}\W_{p}(p(t,x,\D z),p(t,y,\D z))\le(\mathbb{E}(|X(t)-Z(t)|^p))^{1/p},\qquad t\ge0.\end{equation*}
	Since the mapping $t\mapsto |X(t)-Z(t)|$ is  absolutely continuous on $[0,\tau)$,
	we have  $$\frac{\D}{\D t}|X(t)-Z(t)|=\frac{(X(t)-Z(t))(r(Z(t))-r(X(t)))}{|X(t)-Z(t)|},$$ a.e. on $[0,\tau).$
	By assumption, we have
	$$\frac{\D}{\D t}|X(t)-Z(t)|\le-\Gamma\beta(|X(t)-Z(t)|),$$ a.e. on $[0,\tau),$
	which,	together with Lemma \ref{lm:WASS1}, gives $$|X(t)-Z(t)|\leq B_\kappa^{-1}(\Gamma t)\qquad t\geq 0,\quad \kappa\ge |x-y|.$$  For $t\ge\tau$, the term on the left-hand side vanishes, and the term on the right-hand side is well defined and strictly positive ($\beta(t)$ is  convex and $\beta(t)=0$ if, and only if, $t=0$). Now, by taking the expectation and infimum, we conclude $$\W_{p}(p(t,x,\D z),p(t,y,\D z))\leq B_{\kappa}^{-1}(\Gamma t),\qquad t\geq 0,\quad \kappa\ge|x-y|.$$ 
	Next, let  $0<\kappa<|x-y|$. 
	Define  $z_0,\dots,z_{\lceil\delta|x-y|\rceil}\ge0$, where $z_0=x$ and  $$z_{i+1}=z_i+\frac{y-x}{\lceil|x-y|/\kappa\rceil},\qquad i=0,\dots,\lceil|x-y|/\kappa\rceil-1.$$ By construction,
	$|z_0-z_{1}|=\dots=|z_{\lceil|x-y|/\kappa\rceil-1}-z_{\lceil|x-y|/\kappa\rceil}|)\le \kappa$.
	Thus, we conclude that for any $x,y\ge0$ such that $|x-y|> \kappa$, \begin{align*}
	&\W_{p}(p(t,x,\D z),p(t,y,\D z))\\&\leq \W_{p}(p(t,z_0,\D z),p(t,z_1,\D z))+\cdots+\W_{p}(p(t,z_{\lceil|x-y|/\kappa\rceil-1},\D z),p(t,z_{\lceil|x-y|/\kappa\rceil},\D z)).\end{align*} By the previous bound, this implies 
	$$\W_{p}(p(t,x,\D z),p(t,y,\D z))\le \lceil|x-y|/\kappa\rceil B_\kappa^{-1}(\Gamma t)\le \left( \frac{|x-y| }{\kappa}+1\right) B_\kappa^{-1}(\Gamma t),\qquad t\geq0.	$$ 
From this, and using \cite[Lemma 2.3]{Sandric-Arapostathis-Pang-2022} (which implies that for any measure $\upmu\in\mathcal{P}_p$ it holds that $\int_{[0,\infty)}p(t,x,\D y)\upmu(\D x)\in\mathcal{P}_p$), we conclude that for any $\upmu,\upeta\in\mathcal{P}_p$, 
$$\W_p\left(\int_{[0,\infty)}p(t,x,\D y)\upmu(\D x),\int_{[0,\infty)}p(t,x,\D y)\upeta(\D x)\right)\le \left(\frac{\W_p(\upmu,\upeta)}{\kappa}+1\right)B_\kappa^{-1}(\Gamma t),\qquad t\geq 0.$$

We next show $\process{X}$ admits an invariant $\uppi\in\mathcal{P}.$
According to \cite[Proposition 4.3]{Albeverio-Brzezniak-Wu-2010} and \cite[Theorem 3.1]{Meyn-Tweedie-AdvAP-II-1993}, the existence of $\uppi$ will follow if
we show that for each $x\ge0$ and $0<\varepsilon<1$, there is a compact set $C\subset[0,\infty)$ (possibly depending on $x$ and $\varepsilon$) such that $$\liminf_{t\to\infty}\frac{1}{t}\int_0^tp(s,x,C)\D s\ge 1-\varepsilon.$$
By taking $u=0$ in \eqref{eqWASS1}, we have  $$r(v)\ge\Gamma\beta(v),\qquad v\ge0.$$ 
Also, due to the convexity of the function $u\mapsto u^p$ on $[0,\infty)$, we have $$(u+v)^p-u^p\le pv(u+v)^{p-1}\le p(1\vee 2^{p-2})(vu^{p-1}+v^p),\qquad u,v\ge0.$$
Thus, for $\V(u)=u^p$, we have  \begin{align*}\mathcal{L}\V(u)&=-pu^{p-1}r(u)+\int_{[0,\infty)}((u+v)^p-u^p)\upnu(\D v)\\
&\le -p\Gamma u^{p-1}\beta(u)+p(1\vee 2^{p-2}) u^{p-1}\int_{[0,\infty)}v\upnu(\D v)+p(1\vee 2^{p-2}) \int_{[0,\infty)}v^p\upnu(\D v), \qquad u\ge0.\end{align*}
Now, since every super-additive convex function is necessarily non-decreasing and unbounded,
we conclude that there is $u_0>0$
large enough such that 
$$p(1\vee 2^{p-2}) u^{p-1}\int_{[0,\infty)}v\upnu(\D v)+p(1\vee 2^{p-2}) \int_{[0,\infty)}v^p\upnu(\D v)\le \frac{1}{2}p\Gamma u^{p-1}\beta(u),\qquad u\ge u_0,$$ i.e.,
\begin{align*}
&\mathcal{L}\V(u)\\
&\le \left(-p\Gamma u^{p-1}\beta(u)+p(1\vee 2^{p-2}) u^{p-1}\int_{[0,\infty)}v\upnu(\D v)+p(1\vee 2^{p-2}) \int_{[0,\infty)}v^p\upnu(\D v)\right)\Ind_{[0,u_0]}(u)\\
&\ \ \ \ -\frac{1}{2}p\Gamma u^{p-1}\beta(u)\Ind_{[u_0,\infty)}(u)\\
&\le \left(\frac{1}{2}p\Gamma u_=^{p-1}\beta(u_0)+p(1\vee 2^{p-2}) u_0^{p-1}\int_{[0,\infty)}v\upnu(\D v)+p(1\vee 2^{p-2}) \int_{[0,\infty)}v^p\upnu(\D v)\right)\Ind_{[0,u_0]}(u)\\
&\ \ \ \ -\frac{1}{2}p\Gamma u_0^{p-1}\beta(u_0),\qquad u\ge0.
\end{align*}
Clearly, the above relation holds for all $u_1\ge u_0$. Now, according to 
\cite[Theorem 1.1]{Meyn-Tweedie-AdvAP-III-1993}, we conclude that for each $x\ge0$ and $u\ge u_0$, we have \begin{align*}&\liminf_{t\to\infty}\frac{1}{t}\int_0^tp(s,x,[0,u])\D s\\&\ge\frac{\frac{1}{2}p\Gamma u^{p-1}\beta(u)}{\frac{1}{2}p\Gamma u^{p-1}\beta(u)+p(1\vee 2^{p-2}) u^{p-1}\int_{[0,\infty)}v\upnu(\D v)+p(1\vee 2^{p-2}) \int_{[0,\infty)}v^p\upnu(\D v)}.\end{align*}
The assertion now follows by choosing  $u$ large enough.	

Let us now show that any invariant $\uppi\in\mathcal{P}$ of $\process{X}$ satisfies $\uppi\in\mathcal{P}_p$. 
For $\V(u)=u^p$, we have shown that there exist constants $\Gamma_1,\Gamma_2,u_0>0$ such that
  \begin{equation}\label{eq:lyp}\mathcal{L}\V(u)\le \Gamma_{1}\Ind_{[0,u_0]}(u)-\Gamma_{2}u^{p-1}\beta(u),\qquad u\ge0.\end{equation} Now, from \cite[Theorem 4.3]{Meyn-Tweedie-AdvAP-III-1993}, it follows that  $$\int_{[0,\infty)}u^{p-1}\beta(u)\uppi(\D u)\le\frac{\Gamma_{1}}{\Gamma_{2}}$$ for any  invariant  $\uppi\in\mathcal{P}$. The assertion now follows from the properties of $\beta(u)$. Namely, from its convexity and the fact that $\beta(0)=0$, we have $$\beta(ut)\le t\beta(u),\qquad t\in[0,1],\quad u\ge0.$$ In particular, we obtain \begin{equation}\label{eq:con}\beta(1) u\le\beta(u),\qquad u\ge1.\end{equation} Furthermore, by the Markov inequality, we have $$\bar\uppi(u)\le \frac{\int_{[0,\infty)}u^p\uppi(\D u)}{u^p}.$$

Finally, let us prove that $\process{X}$ admits a unique invariant probability measure that satisfies \eqref{eqWASS2}. Let $\uppi,\bar{\uppi}\in\mathcal{P}_p$	be two invariant probability measures of $\process{X}$. Then,  for any $\kappa>0$, we have  \begin{align*}\W_p(\uppi,\bar{\uppi})&=
\W_p\left(\int_{[0,\infty)}p(t,x,\D y)\uppi(\D x),\int_{[0,\infty)}p(t,x,\D y)\bar{\uppi}(\D x)\right)\\&\le \left(\frac{\W_p(\uppi,\bar{\uppi})}{\kappa}+1\right)B_\kappa^{-1}(\Gamma t),\qquad t\ge0.\end{align*} The uniqueness now follows by letting $t\to\infty$.
 Finally, for any $\kappa>0$ and $\upmu\in\mathcal{P}_p$,  we have 
\begin{align*}\W_p\left(\uppi(\D y),\int_{[0,\infty)}p(t,x,\D y)\upmu(\D x)\right)&=
\W_p\left(\int_{[0,\infty)}p(t,x,\D y)\uppi(\D x),\int_{[0,\infty)}p(t,x,\D y)\bar{\upmu}(\D x)\right)\\&
\le\left(\frac{\W_p(\uppi,\upmu)}{\kappa}+1\right)B_\kappa^{-1}(\Gamma t),\qquad t\ge 0,\end{align*} which concludes the proof.	
\end{proof}

A typical example that satisfies the conditions of Theorem \ref{tm:WASS1} is as follows. Let $a\in\R$, $b>0$ and $p,d\ge1$ be arbitrary, and let $r(u)=a-bu^d$, $\beta(t)=t^d$, and let $\upnu$ be any measure on $[0,\infty)$ satisfying $$\int_{[0,\infty)}(u\vee u^p)\upnu(\D u).$$ It is now easy to see that
$r(u)$  satisfies the relation in \eqref{eqWASS1} with $\Gamma=b$. Additionally, if $d=1$, then $B_\kappa^{-1}(t)=\kappa\E^{-t}$, and if $d>1$, then $$\lim_{t\to\infty}\sqrt[d-1]{t}B_\kappa^{-1}(t)=\frac{1}{\sqrt[d-1]{d-1}}.$$
Note also that Theorem \ref{tm:WASS1} covers the case when $A(t)\equiv0$, i.e., a ``zero-input'' storage model. Clearly, since the model is deterministic in this case, $\uppi=\updelta_0$ (the Dirac measure at $0$) and the convergence of $p(t,x,\D y)$ to $\uppi(\D y)$ cannot hold in the total variation distance.

Finally, observe that for $\V(u)=u^p$, from \eqref{eq:lyp} and \eqref{eq:con}, we have $$\mathcal{L}\V(u)\le \Gamma_{1}\Ind_{[0,u_0]}(u)-\Gamma_{2}u^{p},\qquad u\ge0,$$
		for some $u_0,\Gamma_1,\Gamma_2>0.$ Therefore, if additionally $\process{X}$ is $\uppsi$-irreducible, aperiodic and every compact set is petite for $\process{X}$, it follows from \cite[Theorem 5.2]{Down-Meyn-Tweedie-1995}  that $\process{X}$  is geometrically ergodic with respect to the total variation distance. Furthermore, from \cite[Theorem  1.1]{Sandric-Arapostathis-Pang-2022}, it is geometrically ergodic with respect to $\W_1$, and subgeometrically ergodic with rate $t^{1-p/q}$ with respect to $\W_q$ for $q\in[1,p]$.

\section*{Acknowledgement} 
\noindent
MB was supported by EPSRC grant EP/V009478/1 and  CUHK-SZ start-up  UDF01004230. AM was  supported by EPSRC grants EP/V009478/1 and EP/W006227/1 and. NS was supported by \textit{Croatian Science Foundation} project 2277.

\bibliographystyle{alpha}
\bibliography{References}

\end{document}